\documentclass[12pt]{amsart}

\numberwithin{equation}{section}

\usepackage{amsmath,amsthm,amsfonts,eucal,}

\usepackage{subcaption}
    \usepackage{tikz,pgfplots}
    \usepackage{epsfig}
\usepackage[all]{xy}
\usepackage{graphicx}
\usepackage{amssymb}

\def\cf{{\mathcal F}}

\def\ch{{\mathcal H}}

\def\cam{{\mathcal M}}

\def\ga{{\mathfrak A}} 
\def\gb{{\mathfrak B}}
\def\gc{{\mathfrak C}}
\def\gd{{\mathfrak D}}

\def\gs{{\mathfrak S}}

\def\bc{{\mathbb C}}

\def\bn{{\mathbb N}}

\def\br{{\mathbb R}}
\def\bt{{\mathbb T}}

\def\bz{{\mathbb Z}}

\def\a{\alpha}
\def\b{\beta}
\def\g{\gamma}  
\def\d{\delta}  

\def\eps{\varepsilon}

\def\l{\lambda}

\def\n{\nu}

\def\s{\sigma} 
\def\t{\tau}
\def\f{\varphi}  
  
\def\om{\omega} \def\Om{\Omega}

\newtheorem{thm}{Theorem}[section]

\newtheorem{lem}[thm]{Lemma}
\newtheorem{cor}[thm]{Corollary}
\newtheorem{prop}[thm]{Proposition}

\theoremstyle{definition}
\newtheorem{rem}[thm]{Remark}

\def\supp{\mathop{\rm supp}}
\def\Im{\mathop{\rm Im}}

\def\supp{\mathop{\rm supp}}

\def\di{\mathop{\rm d}\!}
\def\di{{\rm d}}

\def\idd{{1}\!\!{\rm I}}

\begin{document}
\title[Periodicity, Type $II_1$ Factors and Free Poisson Laws]
{Periodicity, Type $II_1$ Factors and Free Poisson Laws in Interacting Fock Spaces}

\author{Vitonofrio Crismale}
\address{Vitonofrio Crismale\\
Dipartimento di Matematica\\
Universit\`{a} degli studi di Bari\\
Via E. Orabona, 4, 70125 Bari, Italy}
\email{\texttt{vitonofrio.crismale@uniba.it}}

\author{Yun Gang Lu}
\address{Yun Gang Lu\\
Dipartimento di Matematica\\
Universit\`{a} degli studi di Bari\\
Via E. Orabona, 4, 70125 Bari, Italy}
\email{\texttt{yungang.lu@uniba.it}}

\author{\'Eric Ricard}
\address{\'Eric Ricard\\
UNICAEN, CNRS, LMNO\\ 14000 Caen, France}
\email{\texttt{eric.ricard@unicaen.fr}}

%\date{\today}

\begin{abstract}
We show that the von Neumann algebra generated by position operators in a 2-periodic interacting Fock space is a type $II_1$ factor. 
On the probabilistic side, we prove that the squared position operators have a Marchenko--Pastur distribution with respect to the vacuum state, yielding a natural realization of free Poisson laws within this framework. 

\end{abstract}

\maketitle

\section{Introduction}
The operator algebras generated by creation and annihilation operators on Fock-type spaces provide a rich source of examples at the interface of operator algebras and non commutative probability. Classical instances include the free Fock space and its associated free group factors, as well as their deformations arising from generalized commutation relations. In many situations, algebraic modifications of the underlying Fock structure produce new von Neumann algebras together with remarkable probabilistic phenomena reflected in the distributions of the corresponding position operators.\\
Interacting Fock spaces, introduced by Accardi, Lu and Volovich in connection with stochastic limits in quantum theory \cite{ALV}, provide a general framework for studying weighted deformations of the full Fock space. Since their introduction, they have found applications in several areas of non commutative probability. Connections with monotone probability were established in \cite{Lu}, while more recent developments relate interacting Fock spaces to subproduct systems \cite{GS} and shift-invariant states \cite{CDR}. In the one-dimensional setting, interacting Fock spaces are closely related to orthogonal polynomials through the Jacobi parameters of probability measures \cite{AB}; further developments and generalizations may be found in \cite{AL,AHL1,AHL2}.
The case when the underlying Hilbert space $\mathcal H$ has dimension at least 2 and interacting functions are almost everywhere constant is commonly referred to as the \emph{one-mode type} or \emph{multi-mode} setting \cite{ACL, CL}, in order to distinguish it from the Accardi-Bo\.{z}ejko case. \\
Roughly speaking, a one-mode type interacting Fock space is obtained from the usual full Fock space by deforming the inner product on the $n$-particle subspace by means of a sequence of positive weights $(\lambda_n)_n$. If the ratio sequence $\omega_n := \lambda_n / \lambda_{n-1}$ is constant, the construction reduces, up to isomorphism, to the full Fock space. When the underlying Hilbert space is separable infinite dimensional, the $C^*$-algebra generated by annihilation operators provides a natural generalization of the Fock representation of the Cuntz algebra $\mathcal{O}_\infty$ \cite{C}. In this framework, Xu in \cite{Xu} proved that whenever the underlying Hilbert space has dimension at least two, the vacuum state is tracial on the von Neumann algebra generated by position operators
if and only if $(\omega_n)_n$ is constant.
Recall that the constant case corresponds to the $q$-commutation
relations \cite{BS,BKS} for $q = 0$, or equivalently the free group algebra with ${\rm dim} \mathcal H$ generators. In this setting, the
  resulting von Neumann algebra is well known to be a factor (when ${\rm dim} \mathcal H\geq 2$), and this is also true for the $q$-deformation \cite{R}.\\
  It is therefore natural to investigate what happens beyond the
constant case, where any tracial state, if it exists, cannot coincide
with the vacuum state. The simplest generalization arises when the
sequence $(\omega_n)_n$ takes only two distinct values. This was first investigated in \cite{R2}, where the case
  $\omega_1=1$, and $\omega_n=t$, $n>1$ was considered. Another interesting example can be realized by fixing positive values for both $\omega_{2n}$ and $\omega_{2n+1}$. In this case, the $*$-algebra generated by annihilation operators, as well as its self-adjoint part, becomes $\mathbb{Z}_2$-graded, and for this reason we refer to it as the \emph{period-2} setting.\\
Our main goal is to understand how periodicity influences both the structure of the generated von Neumann algebra and the spectral distributions of the position operators. In particular, we show that period-2 interacting Fock spaces provide a natural setting in which free probabilistic structures arise, even though position operators are not free. Moreover, the von Neumann algebra generated by position operators is a factor of type $II_1$, with the unique normal tracial state given by a convex combination involving the vacuum.\\
Our setting  deals with period-2 interacting Fock spaces with separable infinite dimensional underlying Hilbert space. Here, we show that the $C^*$-algebra generated by annihilation operators is irreducible, and we provide a space-free characterization of it. The unit circle $\mathbb{T}$ acts on the universal $C^*$-algebra by $*$-automorphisms, and this action is equivariant with respect to the natural action of $\mathbb{T}$ on the Fock representation, induced by conjugation with unitaries on the Fock space. As in the Cuntz case, one finds that the Fock representation is, up to isomorphism, the unique faithful covariant $*$-representation of the universal $C^*$-algebra (Theorem~\ref{spacefree}). Since the projection onto the vacuum does not belong to the image of the Fock representation, separability and irreducibility imply that the universal $C^*$-algebra is not of type $I$.\\
For the von Neumann algebra generated by position operators, we first prove that the vacuum vector is cyclic and separating for the even subalgebra with respect to the $\mathbb{Z}_2$-grading, and that this subalgebra is a type $II_1$ factor with the vacuum state as its unique trace. The existence of a unique tracial state on the odd subalgebra requires more delicate arguments, and marks the most difficult problem in Section \ref{vna}, involving two technical lemmas that reflect the different behavior of the kernel of the position operator depending on the order relation between $\omega_{2n}$ and $\omega_{2n+1}$. A suitable convex combination of the corresponding marginal normal traces, induced by the commutation relations, yields a unique normal trace on the whole algebra. Consequently, the full von Neumann algebra is a type $II_1$ factor (Theorem~\ref{factorr}).\\
In the final part of the paper, we study the vacuum distribution of the position operators. Using combinatorial techniques adapted to the interacting Fock space structure in the period-2 case, we show that the distribution of the squared position operator belongs to the Marchenko-Pastur family \cite{MP}, whose parameters depend only on the two periodic weights (Theorem \ref{MarPa}). As is well known, this family coincides with the free Poisson laws of free probability \cite{NS} and includes, as a special case, the square of the (standard) Wigner distribution. In our framework, this is consistent with the fact that the vacuum distribution of the position operator reduces to a (not necessarily standard) Wigner law when the weights $(\omega_n)_n$ are constant (Proposition \ref{distri}).\\
Furthermore, we show that the right endpoint of the support of the vacuum distribution coincides with the norm of each position operator (Corollary~\ref{norm}), even though the vacuum vector is not separating for the entire von Neumann algebra. \\
We also investigate independence properties of the position operators from the viewpoint of free probability. We prove that freeness admits a sharp characterization in terms of the weight sequence: in the vacuum state, the position operators are free if and only if the weights are constant, while their squares (and hence all even powers) form a free family if and only if the sequence is 2-periodic (Proposition \ref{2perfree}). In particular, periodicity recovers free probabilistic structures after passing to even observables. As a consequence, the von Neumann algebra generated by the position operators contains a copy of the free group factor $L(\mathbb{F}_\infty)$, and one can identify free additive and multiplicative convolution structures in this framework.
These results show that periodicity constitutes a new mechanism through which type $II_1$ factors and free-probabilistic structures emerge in interacting Fock spaces, extending several phenomena previously known only in the constant-weight case.

\section{Preliminaries}

Let $S$ be a finite set with cardinality $|S|$. A partition $\pi$ of $S$ is a collection of mutually disjoint subsets $B_1,\ldots,B_k$, whose union is $S$. The elements of $\pi$ are called \emph{blocks}. Throughout the paper we set $S=[m]:=\{1,\ldots,m\}$, and we denote by $P(m)$ the set of all partitions of $[m]$.\\
A partition $\pi$ is called a \emph{pair partition} if each block has cardinality $2$, that is, $|B_j|=2$ for all $j=1,\ldots,k$. In this case, the total number of elements must be even, and the set of pair partitions of $[2n]$ is denoted by $PP(2n)$. It is convenient to represent each block in this case by its left and right endpoints, writing
\[
B_j=(l_j,r_j), \qquad l_j<r_j\,,
\]
and use the convention that $l_j<l_{j+1}$ for each $j=1,\ldots, n-1$. Therefore, automatically $l_1=1$ for each $\pi\in PP(2n)$.\\
A partition $\pi$ is said to have a \emph{crossing} if there exist two distinct blocks $B_i$ and $B_j$ and elements $v_1,v_2\in B_i$ and $w_1,w_2\in B_j$ such that
\[
v_1 < w_1 < v_2 < w_2.
\]
If no such configuration exists, the partition is called \emph{non crossing}. We denote by $NCPP(2n)$ the set of non crossing pair partitions of $[2n]$. Thus, for $\pi=\{(l_h,r_h)\}_{h=1}^{n}\in NCPP(2n)$, one has $(l_n,r_n)=(l_n,l_n+1)$. \\
The set $NCPP(2n)$ can be naturally identified with a subset of $\{-1,1\}$-valued maps on $[2n]$. More precisely, let $\{-1,1\}^{2n}$ denote the set of maps $\varepsilon:[2n]\to\{-1,1\}$, and define
\[
\{-1,1\}^{2n}_+ :=
\left\{
\varepsilon\in\{-1,1\}^{2n}\ \middle|\
\sum_{h=1}^{2n}\varepsilon(h)=0
\ \text{and}\
\sum_{h=1}^{r}\varepsilon(h)\le 0\,, r\in[2n]
\right\}.
\]
Then the correspondence
\[
NCPP(2n)\ni\{(l_h,r_h)\}_{h=1}^n \mapsto \varepsilon\in\{-1,1\}^{2n}_+,
\]
defined by setting $\varepsilon(l_h)=-1$ and $\varepsilon(r_h)=1$ for each $h$, establishes a bijection.\\
Removing one block, say $(l_k,r_k)$, from a partition $\{(l_h,r_h)\}_{h=1}^n$ belonging to $NCPP(2n)$ gives rise to a non crossing pair partition $\{(l'_h,r'_h)\}_{h=1}^{n-1}$ in $NCPP(2n-2)$. Indeed, if $d:\{(l_h,r_h)\mid h=1,\ldots, n, h\neq k\}\rightarrow [2n-2]$ is the order preserving map, one denotes $l'_h=d(l_h)$, $r'_h=d(r_h)$ if $h<k$ and $l'_h=d(l_{h+1})$, $r'_h=d(r_{h+1})$ if $h>k$. The new partition is referred to as the \emph{reduction} of the original partition. Under the above correspondence, this operation yields a reduced sequence $\varepsilon'\in\{-1,1\}^{2n-2}_+$.

\medskip

Let $\mu$ be a probability measure on $\mathbb{R}$ with finite moments, and denote by $(m_n(\mu))_{n\geq 0}$ its moment sequence. The associated \emph{moment generating function} is defined, for $z\in\mathbb{C}$, by
\[
\mathcal{M}_\mu(z):=\sum_{n=0}^{\infty} z^n m_n(\mu),
\]
which is understood as a formal power series whenever the series does not converge absolutely.\\
Denote by $\mathbb{C}^+$ and $\mathbb{C}^-$ the upper and lower complex half-planes, respectively. The \emph{Cauchy transform} of $\mu$ is defined by
\[
\mathcal{G}_{\mu}(z):=\int_{-\infty}^{+\infty}\frac{\mu(dx)}{z-x}.
\]
The function $\mathcal{G}_{\mu}$ is analytic on $\mathbb{C}\setminus\supp(\mu)$, satisfies $\mathcal{G}_{\mu}(\overline{z})=\overline{\mathcal{G}_{\mu}(z)}$, and maps $\mathbb{C}^+$ into $\mathbb{C}^-$. As a consequence, it suffices to consider $\mathcal{G}_{\mu}$ on $\mathbb{C}^+\cup\bigl(\mathbb{R}\setminus\supp(\mu)\bigr)$. In this region, the Cauchy transform uniquely determines the measure $\mu$ (see, \textit{e.g.}, \cite{HO}):

\medskip
\noindent
(1) The limit
\begin{equation}
\label{stiedef}
\rho_\mu(x):=-\frac{1}{\pi}\lim_{y\to 0^{+}}\Im \mathcal{G}_{\mu}(x+iy)
\end{equation}
exists for almost every $x\in\mathbb{R}$, and $\rho_\mu(x)\,dx$ coincides with the absolutely continuous part of $\mu$. Formula \eqref{stiedef} is known as the \emph{Stieltjes inversion formula}.

\medskip
\noindent
(2) The function $\mathcal{G}_{\mu}$ has a simple pole at $z=a\in\mathbb{R}$ if and only if $a$ is an isolated point of $\supp(\mu)$. In this case,
\[
\mu=c\,\delta_a+(1-c)\nu,\qquad 0\le c\le1,
\]
where $\nu$ is a probability measure satisfying $\nu(\{a\})=0$. Moreover, $c=\operatorname*{Res}_{z=a}\mathcal{G}_{\mu}(z)$.\\
A direct computation shows that
\begin{equation}
\label{AA}
\mathcal{G}_{\mu}(z)
=\frac{1}{z}\,
\mathcal{M}_\mu\!\left(\frac{1}{z}\right),
\qquad z\in\mathbb{C}^+.
\end{equation}
We conclude this section by recalling the Marchenko-Pastur distribution $\mu$ (see, \textit{e.g.}, \cite{MP}). In free probability, this distribution is also known as the \emph{free Poisson distribution}, since it arises as the Poisson limit law for two-point processes with respect to additive free convolution (see, \textit{e.g.}, \cite{NS}). For parameters $\lambda\geq 0$ and $\gamma\in\mathbb{R}$, the measure $\mu$ is defined by
\[
\mu :=
\begin{cases}
(1-\lambda)\delta_0+\widetilde{\mu}, & 0\le \lambda\le1,\\
\widetilde{\mu}, & \lambda>1,
\end{cases}
\]
where
\[
d\widetilde{\mu}(t)
:=\frac{1}{2\pi\gamma t}
\sqrt{4\lambda\gamma^2-(t-\gamma(1+\lambda))^2}\,
\chi_{[\gamma(1-\sqrt{\lambda})^2,\ \gamma(1+\sqrt{\lambda})^2]}(t)\,dt.
\]
In analogy with the classical setting, $\lambda$ is called the \emph{rate} and $\gamma$ the \emph{jump size}.\\
The Cauchy transform of $\mu$ is given by
\begin{equation}
\label{CaMP}
\mathcal{G}_{\mu}(z)=
\frac{z+\gamma-\lambda\gamma-\sqrt{(z-\gamma(1+\lambda))^2-4\lambda\gamma^2}}
{2\gamma z}.
\end{equation}
Finally, note that $\mu$ coincides with the law of the square of a Wigner random variable with variance $\sigma^2$ in the special case $\lambda=1$ and $\gamma=\sigma^2$.

\section{On one-mode type interacting Fock spaces}
In this section we introduce one-mode type interacting Fock spaces and study vacuum moments of the associated creation and annihilation operators.\\

Let $\ch$ be a separable Hilbert space with canonical basis $(e_n)_n$, and fix a sequence $(\l_n)_n$ of non-negative numbers with $\l_0=1$. For each $n$, $\l_n$ naturally induces a deformation of the usual tensor product $(\ch^{\otimes n},(\cdot,\cdot)_n)$ by defining
\[
\langle f_1\otimes \cdots \otimes f_n,\, g_1\otimes \cdots \otimes g_n\rangle_n
:= \l_n (f_1\otimes \cdots \otimes f_n,\, g_1\otimes \cdots \otimes g_n)_n .
\]
For any $n$, denote by $\ch_n$ the Hilbert space completion with respect to this new inner product, and note that $\ch_0:=\bc \Omega$, for a distinguished unit vector $\Omega$, called the vacuum. The gradation
\[
\cf(\ch,\l_n):=\bigoplus_{n=0}^{\infty} \ch_n
\]
defines a Hilbert space called \emph{one-mode type interacting Fock space}, see, for example, \cite{ACL}. If each $\l_n$ is equal to $1$, one recovers the usual full Fock space.\\
The creation operator with test function $f\in\ch$ is defined by
\[
a^\dag(f)\Omega := f, \qquad
a^\dag(f)\, f_1\otimes \cdots \otimes f_n
:= f\otimes f_1\otimes \cdots \otimes f_n\,.
\]
where $f_1,\ldots, f_n\in\ch$.
Note that if there exists an integer $n$ such that $\l_n=0$, then one must have $\l_{n+m}=0$ for all $m$: this condition ensures that the creation operator is well defined. The annihilation operator is given by
\[
a(f)\Omega := 0, \qquad
a(f)\, f_1\otimes \cdots \otimes f_n
:= \om_n \langle f,f_1\rangle f_2\otimes \cdots \otimes f_n ,
\]
where $\om_n := \frac{\l_n}{\l_{n-1}}$. When $(\om_n)_n$ is bounded, both operators extend linearly to the whole space $\cf(\ch,\l_n)$ and are adjoint to each other. To simplify notation, in what follows we write $a_i$ and $a_i^\dag$ instead of $a(e_i)$ and $a^\dag(e_i)$, respectively.\\
If $N$ denotes the number operator, then for each $f,g\in\ch$, one has
\begin{equation}
\label{cr}
a(f) a^{\dag} (g) = \langle f, g\rangle\,\om_{N+1},
\end{equation}
where $\d_{i,j}$ is the Kronecker symbol and $\om_{N+1}$ is determined by $N$ and the sequence $(\om_n)_n$ via functional calculus. In addition, one has
\begin{equation}
\label{cr2}
\om_{N+1} a^\dag(f) = a^\dag(f) \om_{N+2}, \qquad
\om_{N+1} a(f) = a(f) \om_N .
\end{equation}
Since the subject of these notes is a particular one-mode type interacting Fock space in which the sequence $(\om_n)_n$ takes only two values infinitely many times, from now on we assume that $\om_n>0$ for each $n$ and $(\om_n)_n$ is bounded. Equivalently, the latter condition gives the creation and annihilation operators are bounded on $\cf(\ch,\l_n)$. The truncated case $\om_n=0$ for some $n$ has been treated in \cite{CDR}.\\
As in \cite{AHL2}, one observes that, for any $j\in\bn$, each finite sequence of $a_j$ and $a_j^\dag$ can be written in the so-called normal (or Wick) order, that is, with all creation operators appearing to the left of all annihilation operators. More precisely, \eqref{cr}, \eqref{cr2} and an induction argument yield that for each $n\in\bn$, and $\s(1),\ldots \s(n)\in\{1,\dag\}$
\begin{equation*}
a_j^{\s(1)}a_j^{\s(2)}\cdots a_j^{\s(n)}=(a_j^\dag)^{n_+}a_j^{n_-}\prod_{k=1}^{n_0}\om_{N+h_k}\,,
\end{equation*}
where $n_0$ is the number of couples in the sequence of the type appearing in \eqref{cr}, $n_+$ ($n_-$) is the number of creators (annihilators) in the sequence minus $n_0$, and finally $h_1,\ldots, h_{n_0}\in\bz$. When the sequence contains an even number of operators, say $2n$, then $n_++\,n_-=2(n-n_0)$, that is $n_+ +\,n_-\in\{2m\mid m\in[0,n]\}$, and consequently $n_+-n_-\in\{2m \mid m\in[-n,n]\}$. The following proposition will be exploited in Section \ref{seclaw}.
We use the convention that $\omega_n=0$ if $n<0$.
\begin{prop}\label{free1}
For each $j\in\bn$, let $\gd^0_j$ be the linear span of
$$
\bigg\{(a_j^\dag)^{n_+}a_j^{n_-}\prod_{k=1}^{n_0}\om_{N+h_k} :n_0,n_-,n_+\in\bn,\, n_++n_-\in 2\bn,\, h_k\in\bz,\, k\in [n_0]\bigg\}.
$$
Then $\gd^0_j$ is a $*$-algebra containing the identity operator of $\cf(\ch,\l_n)$.
\end{prop}
\begin{proof}
For any $j\in\bn$, $\gd^0_j$ is closed under involution, and, if $\prod_{k\in\emptyset}:=1$, then $\displaystyle{I_{\cf(\ch,\l_n)}=(a_j^\dag)^{0}a_j^{0}\prod_{k\in\emptyset}\om_{N+h_k}}$. It remains to prove that, if
$$
b:=(a_j^\dag)^{n_+}a_j^{n_-}\prod_{k=1}^{n_0}\om_{N+h_k}\,, \quad c:=(a_j^\dag)^{m_+}a_j^{m_-}\prod_{k=1}^{m_0}\om_{N+r_k}\,,
$$
where $n_++\,n_-, m_+\,+m_-\in 2\bn$ and $h_k,  r_k\in\bz$ for each $k\in [1,n_0]$, then $bc\in\gd^0_j$. Indeed, \eqref{cr2} gives
$$
bc=(a_j^\dag)^{n_+}a_j^{n_-}(a_j^\dag)^{m_+}a_j^{m_-}\prod_{k=1}^{n_0}\om_{N+h_k+m_++m_-}\prod_{k=1}^{m_0}\om_{N+r_k}\,.
$$
A repeated use of \eqref{cr} and \eqref{cr2} yields
$$
a_j^{n_-} (a_j^\dag)^{m_+}=\begin{cases}
\displaystyle{(a_j^\dag)^{(m_+-n_-)}\prod_{k=1}^{n_-}\om_{N+k+m_+-n_-}}
&\text{ if }n_-\leq m_+\\
\displaystyle{a_j^{(n_--m_+)}\prod_{k=1}^{m_+}\om_{N+j}}
&\text{ if }n_->m_+\,.
\end{cases}
$$
This means that $bc$ is
$$
\displaystyle{(a_j^\dag)^{(n_++m_+-n_-)}a_j^{m_-}
\prod_{k=1}^{n_-}\om_{N+k+m_+-n_--m_-}
\prod_{k=1}  ^{n_0} \om_{N+h_k+m_+-m_-}\prod_{k=1}  ^{m_0} \om_{N+r_k}}
$$
or
$$
\displaystyle{(a_j^\dag)^{n_+}a_j^{n_--m_++m_-}
\prod_{k=1}^{m_+}\om_{N+k-m_-}
\prod_{k=1}^{n_0} \om_{N+h_k+m_+-m_-}\prod_{k=1}  ^{m_0} \om_{N+r_k}}
$$
according to whether $n_-\leq m_+$ or $n_->m_+$, respectively.\\ The assumptions on $n_+,n_-,m_+-m_-$ imply that, when $n_-\leq m_+$, then $n_+-n_-+m_++m_-\in 2\bn$, whereas, if $n_->m_+$, one finds $n_++n_--m_++m_-\in 2\bn$.
\end {proof}
The self-adjoint part of the creation and annihilation operators, usually referred to as the Gaussian operator, is defined by $s_j := a_j + a_j^\dag$, for $j\geq 1$. The vector state associated with $\Om$ is usually referred to as the vacuum state or the vacuum expectation, and often denoted by $\t_\Om$.\\
Recalling that for each natural integers $j$ and $n$ one has
\begin{align}
\label{2ome01b}
\t_\Om(s_j^n)=\langle\Omega,s_j^n\Omega\rangle
=
\begin{cases}
0 & \text{if $n$ is odd},\\
\displaystyle
\sum_{\eps\in\{-1,1\}^{2n}_+}
\langle\Omega,
a_j^{\eps(1)}\cdots a_j^{\eps(2m)}\Omega\rangle
& \text{if $n=2m$},
\end{cases}
\end{align}
where
\begin{equation*}
%\label{epsi}
a_j^\eps :=
\begin{cases}
a_j & \text{if }\eps=-1,\\
a_j^\dag & \text{if }\eps=1,
\end{cases}
\end{equation*}
we state a result that will be needed later.
\begin{prop}
\label{moment1}
For each $\eps\in\{-1,1\}^{2n}_+$ and each $j\in\bn$, let $\{l_h,r_h\}_{h=1}^n$ be the corresponding non crossing pair partition associated with $\eps$. If $f,g,f_1,\ldots, f_{2n}\in\ch$, then
\begin{equation}
\label{2ome02a}
\langle\Omega,
a^{\eps(1)}(f_1)\cdots a^{\eps(2n)}(f_{2n})\Omega\rangle
=
\prod_{h=1}^n \om_{2h-l_h}\langle f_{l_h},f_{r_h}\rangle,
\end{equation}
and
\begin{align}
\label{2ome02b}
\langle\Omega,
a(f) a^{\eps(1)}(f_1)\cdots a^{\eps(2n)}(f_{2n}) a^\dag(g) \Omega\rangle
=
\om_1\langle f,g\rangle \prod_{h=1}^n \om_{2h-l_h+1}\langle f_{l_h},f_{r_h}\rangle.
\end{align}
\end{prop}
\begin{proof}
We fix $j\in\bn$ and proceed by induction. The case $n=1$ reduces to
\[\langle\Om,a^{\eps(1)}(f_1) a^{\eps(2)}(f_2)\Om\rangle
=\langle\Om,a(f_1) a^\dag(f_2)\Om\rangle=\om_1\langle f_1,f_2\rangle\]
and
\begin{align*}
\langle\Om,a(f)a^{\eps(1)}(f_1) a^{\eps(2)}(f_2)a^\dag(g)\Om\rangle&=\langle\Om,a(f)a(f_1) a^\dag(f_2)a^\dag(g)\Om\rangle\\
&=
\om_1 \om_2\langle f,g\rangle \, \langle f_1,f_2\rangle\,,
\end{align*}
where we exploited \eqref{cr} and \eqref{cr2}.\\
Suppose now that \eqref{2ome02a} and \eqref{2ome02b} hold for any $n\leq m$. We consider the case $n=m+1$. For any $\eps\in\{-1,1\}^{2m+2}_+$, we denote by $\eps'$ the reduction of $\{l_h,r_h\}_{h=1}^{m+1}$ corresponding to the erasing of the most inner block $(l_{m+1}, r_{m+1})$. Then one has
\begin{align}
\begin{split}
\label{2ome02d}
&\langle\Om,a^{\eps(1)}(f_1)\cdots a^{\eps(2m+2)}(f_{2m+2}) \Om\rangle \\
=&\langle\Om,a^{\eps(1)}(f_1)\cdots a^{\eps(l_{m+1}-1)}(f_{l_{m+1}-1}) a(f_{l_{m+1}})a^\dag(f_{r_{m+1}})\cdots a^\dag(f_{2m+2})\Om\rangle \\
=&\om_{2m+2-l_{m+1}}\langle f_{l_{m+1}},f_{r_{m+1}}\rangle \prod_{h=1}^{m} \om_{2h-l'_h}\langle f_{l'_h},f_{r'_h}\rangle
\end{split}
\end{align}
where the last equality comes from the induction assumption. As for each $h=1,\ldots,m$ one has $l'_h=l_h$ and $f_{r'_h}=f_{r_h}$,  \eqref{2ome02d} turns out to be
$$
\langle\Om,a^{\eps(1)}(f_1)\ldots a^{\eps(2m+2)}(f_{2m+2}) \Om\rangle=\prod_{h=1}^{m+1}\om_{2h-l_h}\langle f_{l_h},f_{r_h}\rangle\,.
$$
The same arguments also show that \eqref{2ome02b} holds when $n=m+1$.
\end{proof}

\section{The $C^*$-algebra generated by annihilation operators in period-2 interacting Fock spaces}
In this section, we investigate a class of one-mode type interacting Fock spaces associated with period-2 interacting sequences. Our main goal is to analyze the $C^*$-algebra generated by the annihilation operators and to provide a space-free characterization of it. \\

Let us start with the sequence $(\l_n)_n$ defined, for fixed $\a,\b>0$, by
\begin{equation}\label{lambda}
\l_{2n} := \a^n \b^n, \qquad
\l_{2n+1} := \a^{n+1} \b^n, \quad n\geq 0.
\end{equation}
As a consequence,
\begin{equation}\label{period}
\om_{2n-1} = \a, \qquad \om_{2n} = \b, \quad n \geq 1.
\end{equation}
The corresponding one-mode type interacting Fock space is said to have \emph{period 2} and is denoted by $\cf(\ch,\a,\b)$. The creation and annihilation operators are bounded. Indeed, for any $f\in\ch$, we have
\begin{equation}\label{norm}
\|a^\dag(f)\| = \|a(f)\| \le \max(\sqrt{\a}, \sqrt{\b}) \|f\|.
\end{equation}
Let $\ga_0$ denote the $*$-algebra generated by $\{a_i \mid i \in \bn\}$, and let $\ga$ be its uniform closure. The definition of the creation operators implies that $\Om$ is cyclic for $\ga$.
Note that $\ga_0$ is unital. Indeed, from \eqref{cr} and \eqref{cr2} one obtains
\begin{equation}\label{id}
a_i a_j a_j^\dag a_i^\dag = a_i \om_{N+1} a_i^\dag = \om_{N+2} a_j a_j^\dag = \om_{N+2} \om_{N+1} = \a \b I, \quad i,j \in \bn,
\end{equation}
where $I$ is the identity operator.

Let $P$ be the projection of $\cf(\ch,\a,\b)$ onto $\bigoplus_{n=0}^{\infty} \ch_{2n}$, and define $Q := I-P$. By \eqref{period}, the relations \eqref{cr} and \eqref{cr2} take the form
\begin{align}\label{crperiod}
\begin{split}
&a_i a_j^\dag = \d_{i,j} (\a P + \b Q), \quad i,j\in\bn,\\
&a_i P = Q a_i, \qquad a_i^\dag P = Q a_i^\dag, \quad i \in \bn\,.
\end{split}
\end{align}
The relations above immediately yield a linear basis for $\ga_0$:

\begin{prop}\label{basis}
The words of the form
\begin{align}\label{hambas}
\begin{split}
X &:= a^\dag_{i_1} \cdots a^\dag_{i_n} a_{j_1} \cdots a_{j_m} P,\\
\widetilde{X} &:= a^\dag_{l_1} \cdots a^\dag_{l_p} a_{k_1} \cdots a_{k_q} Q,
\end{split}
\end{align}
where $n,m,p,q \geq 0$ and
$i_1,\dots,i_n,j_1,\dots,j_m,l_1,\dots,l_p,k_1,\dots,k_q \in \bn$, form a Hamel basis of $\ga_0$.
\end{prop}

\begin{proof}
Let $L$ denote the index set such that $(Y_l)_{l \in L}:=(X_l,\widetilde{X}_l)_{l \in L}$ consists of all words of the form \eqref{hambas}, and consider
\[
Y := \sum_{l \in F} \d_l X_l + \l_l \widetilde{X}_l,
\]
where $F \subset L$ is finite and $\d_k,\l_l$ are complex scalars. From \eqref{crperiod}, the set $(Y_l)_{l \in L}$ generates $\ga_0$. Assume that $Y=0$. We may assume that the identity $I$ does not appear in $Y$. Otherwise, denoting by $l_0$ the index such that $Y_{l_0}=I$, one can take a unit vector $e_k$ in the kernel of each word in $Y$ ending with an annihilator. Thus, for $l\neq l_0$, the vector $Y_l e_k$ is orthogonal to $e_k$ whenever it is nonzero, and consequently $\l_{l_0} = \langle Y e_k, e_k \rangle = 0$. Similarly, $\d_{l_0}=\langle Y (e_k\otimes e_k), e_k\otimes e_k \rangle = 0$. \\
Now, for any $l \in L$, write explicitly
\[
X_l = a^\dag_{i_{l_1}} \cdots a^\dag_{i_{l_n}} a_{j_{l_1}} \cdots a_{j_{l_m}} P, \qquad
\widetilde{X}_l = a^\dag_{i_{l_1}} \cdots a^\dag_{i_{l_n}} a_{j_{l_1}} \cdots a_{j_{l_m}} Q,
\]
and choose vectors
\[
\xi_l := e_{j_{l_m}} \otimes\cdots \otimes e_{j_{l_1}}, \qquad
\eta_l := e_{i_{l_1}} \otimes \cdots \otimes e_{i_{l_n}},
\]
with the convention that $\xi_l=\Omega$ if $m=0$ and $\eta_l=\Omega$ if $n=0$. Then, depending on whether $m = 2r$ or $m = 2r+1$, we have
\begin{align*}
0 = \langle Y \xi_l, \eta_l \rangle &=
\begin{cases}
\d_l \langle X_l \xi_l, \eta_l \rangle, & \text{if } m = 2r,\\
\d_l \langle \widetilde{X}_l \xi_l, \eta_l \rangle, & \text{if } m = 2r+1,
\end{cases}\\
&=
\begin{cases}
\a^r \b^r \l_n \d_l, & \text{if } m = 2r,\\
\a^{r+1} \b^r \l_n \d_l, & \text{if } m = 2r+1.
\end{cases}
\end{align*}
Since $\a,\b>0$ and $\l_n>0$ by \eqref{lambda}, it follows that $\d_l = 0$ in all cases.
\end{proof}
The projection onto the vacuum is denoted by $P_\Omega$. As shown below, it does not belong to $\ga$.
\begin{prop}\label{novacuum}
The $C^*$-algebra $\ga$ does not contain $P_\Omega$.
\end{prop}
\begin{proof}
We argue by contradiction. Suppose $P_\Omega \in \ga$. Then, for any $\eps>0$, there exists $X_\eps \in \ga_0$ such that $\|X_\eps - P_\Omega\| < \eps$. By Proposition \ref{basis}, each word in $\ga_0$ can be put in normal order; thus, for a finite set $F \subset \bn$,
\[
X_\eps = \sum_{l \in F} a^\dag_{i_{l_1}^\eps} \cdots a^\dag_{i_{l_n}^\eps} a_{j_{l_1}^\eps} \cdots a_{j_{l_m}^\eps} (\d_l^\eps P + \l_l^\eps Q),
\]
where $\d_l^\eps$ and $\l_l^\eps$ are complex numbers. Note that $X_\eps$ contains $I$. Indeed, otherwise $\eps>\|(X_\eps-P_\Om)\Om\|=\|\xi_\eps-\Om\|$, where $\xi_\eps\perp \Om$, which yields the contradiction $\eps>1$.\\
Let $l_0\in F$ be the index corresponding to $n=m=0$. Then $\|(X_\eps - P_\Omega) \Omega\| = \left\|  \d_{l_0}^\eps \Omega - \Omega + \xi_\eps \right\|$,
where $\xi_\eps$ is orthogonal to $\Omega$. Hence,
\begin{equation}\label{absurd1}
\left| 1 - \d_{l_0}^\eps \right| < \eps.
\end{equation}
On the other hand, let $m$ be larger than all indices of creators and annihilators appearing in $X_\eps$. Applying $X_\eps - P_\Omega$ to $e_m \otimes e_{m+1}$ yields $\| \d_{l_0}^\eps (e_m \otimes e_{m+1}) + \eta_\eps \| < \eps$,
where $\eta_\eps \perp e_m \otimes e_{m+1}$. Therefore, $|\d_{l_0}^\eps| < \eps$,
contradicting \eqref{absurd1}.
\end{proof}

\begin{prop}\label{irre}
The $C^*$-algebra $\ga$ is irreducible.
\end{prop}

\begin{proof}
We prove that $\ga' = \bc I$. First, note that $P_\Omega \in \ga''$. Indeed, for any $i \in \bn$ and $\xi \in \cf(\ch,\l_n) \ominus \bc \Omega$, \eqref{period} gives
\[
\left(a_i a_i^\dag + \sum_{j \in \bn} a_j^\dag a_j \right) (\xi \oplus \Omega) = (\a + \b) \xi \oplus \a \Omega,
\]
so that
\[
\b P_\Omega = (\a + \b) I - a_i a_i^\dag - \sum_{j \in \bn} a_j^\dag a_j.
\]
If $b \in \ga'$, then $b \Omega = b P_\Omega \Omega = P_\Omega b \Omega = \langle \Omega, b \Omega \rangle \Omega$. Since $\Omega$ is cyclic for $\ga''$, it follows that $b = \t_\Om(b) I$.
\end{proof}
It is possible to provide a space-free construction for the $C^*$-algebra $\ga$. More precisely, for fixed $\a,\b>0$, let $\gb_0$ be the universal unital $*$-algebra generated by a system $\{b_i \mid i \in \bn\}$ realizing the following relations:
\begin{align}
\label{ucr1} b_i b_j^* &= \d_{i,j} b_k b_k^*, & i,j,k \in \bn,\\
\label{ucr2} b_i b_i^* - \b \idd &= \frac{1}{\a-\b} (b_i b_i^* - \b \idd)^2, & i \in \bn,\\
\label{ucr3} \frac{1}{\a-\b} b_j (b_i b_i^* - \b \idd) &= b_j - \frac{1}{\a-\b} (b_i b_i^* - \b \idd) b_j, & i,j \in \bn.
\end{align}
where $\idd$ is the unit of $\gb_0$.
Observe that \eqref{ucr1}-\eqref{ucr2} imply that the self-adjoint element
\begin{equation}\label{pfree}
p := \frac{1}{\a-\b} (b_i b_i^* - \b \idd)
\end{equation}
is an idempotent independent of $i$. Then \eqref{ucr3} can be written simply as
\begin{equation}\label{pfree2}
b_j p = (\idd - p) b_j\,, \quad j \in \bn\,.
\end{equation}
Note that it also yields $(1-p)b_j=b_jp$.

\begin{thm}\label{spacefree}
The universal $C^*$-algebra $\gb$ generated by $\{b_i \mid i \in \bn\}$ satisfying \eqref{ucr1}-\eqref{ucr3} exists and is isomorphic to $\ga.$%the %map
%\[
%\Phi: b_i \in \gb \mapsto a_i \in \ga
%\]
%is an isomorphism.
\end{thm}
\begin{proof}
Consider the maximal $C^*$-seminorm on $\gb_0$ given by   $\|\cdot\|_{\max}:=\sup\{\|\pi(\cdot)\|\mid \pi\, \text{$*$-representation}\}$.
Since \eqref{ucr1} and \eqref{pfree} are equivalent to
\begin{equation}\label{ucr5}
b_ib^*_j=\d_{i,j}(\a p+\b(\idd-p))\,, \quad i,j\,,
\end{equation}
the map $\Phi_0:\gb_0\rightarrow \ga_0$ defined by $\Phi_0(b_i) = a_i$ preserves \eqref{crperiod} and thus extends by universality to a $*$-homomorphism. Pick now the words in $\gb_0$ of type
\begin{equation}\label{hamuni}
x = b^*_{i_{1}} \cdots b^*_{i_{n}} b_{j_{1}} \cdots b_{j_{m}} p, \qquad
\widetilde{x} = b^*_{i_{1}} \cdots b^*_{i_{n}} b_{j_{1}} \cdots b_{j_{m}} (\idd - p),
\end{equation}
with $n,m \geq 0$, $i_1,\ldots,i_n, j_1,\ldots,j_m\in\bn$. As usual, $x=p$ and $\widetilde{x}=\idd-p$ when $n=m=0$. Using \eqref{pfree} and \eqref{pfree2}, one obtains $b_i b_j b_j^* b_i^* = \a \b \idd$ for each $i,j\in\bn$. Therefore, $\Phi_0$ is unital by \eqref{id} and sends $p$ to $P$. This shows that $\Phi_0(x)=X$ and $\Phi_0(\widetilde{x})=\widetilde{X}$, where $X$ and $\widetilde{X}$ are given in \eqref{hambas}.  Hence, linear combinations of words of type $x$ and $\widetilde{x}$ make up a Hamel basis in $\gb_0$. Thus, $\Phi_0$ is injective and $\|\cdot\|_{\max}$ is actually a norm. \\
Now let $\pi$ be a $*$-representation of $\gb_0$. By \eqref{ucr5}, for each $i\in\bn$, $\|\pi(b_i)\|^2=\|\pi(b_ib_i^*)\|= \|\a \pi(p)+ \b\pi(\idd-p)\|=\max\{\a,\b\}$, since $\pi(p)$ and $\pi(\idd-p)$ are orthogonal. Consequently, for each $y\in\gb_0$, there exists a constant $C_y$ such that $\|\pi(y)\|\leq C_y$, and therefore $\|\cdot\|_{\max}$ is finite.\\
The universal $C^*$-algebra $\gb$ is then the completion of $\gb_0$ with respect to the maximal norm, and $\Phi_0$ extends to a $*$-epimorphism $\Phi: \gb \to \ga$.\\
Proving that $\Phi$ is isometric requires more work. We follow standard arguments based on the gauge action.\\
Let $\bt:=\{z\in\bc\mid |z|=1\}$ be the unit circle: For each  $z\in\bt$, define the map $\g_z(b_i):=z b_i$ for each $i\in\bn$. As $\g_z(b_i)$ satisfies \eqref{ucr1}-\eqref{ucr3}, $\g_z$ extends uniquely by universality to a $*$-homomorphism which is also bijective with inverse $\g_{\overline{z}}$. This gives a strongly continuous action of $\bt$ on $\gb$ by $*$-automorphisms. This action is referred to as the gauge action, and the gauge-invariant subalgebra $\gb^\g:=\{b\in\gb\mid \g_z(b)=b,\,\,z\in\bt\}$ is algebraically spanned by the words in the Hamel basis of $\gb_0$ containing the same number of $b_i$ and $b_i^*$. Averaging this action with respect to the Haar-Lebesgue measure yields the map $T$ on $\gb$ such that $b\in\gb\mapsto T(b):=\int_{\mathbb T} \g_z(b) dz$. Note that for each $x$ or $\widetilde{x}$ in \eqref{hamuni}, one has $T(x)=\d_{n,m}x$, $T(\widetilde{x})=\d_{n,m}\widetilde{x}$. More precisely, $T$ is a faithful conditional expectation onto $\gb^\g$, and consequently for any $b\in \gb$, $\|b\|_{\max}=\lim_{n\to \infty} \|T((b^*b)^n)\|_{\max}^{1/2n}$.\\
Similarly, the family of unitaries $U_z$ on the Fock space $\cf(\ch,\a,\b)$ such
that $U_z\lceil_{\ch_n}:=z^n I_{\ch_n}$ for each $z\in\bt$ and
$n\in\bn$, implement a gauge representation $\widetilde{\g}_z$ of
$\bt$ on $\ga$, namely $\widetilde{\g}_z(b):=U_zbU_z^*$, with
$z\in\bt$ and $b\in\gb$. Let us denote by $\widetilde{T}$ the faithful
conditional expectation onto the fixed-point subalgebra
$\widetilde{T}(\cdot):=\int_{\bt}\widetilde{\g}_z(\cdot)dz$. As the
two actions described above are equivariant, one has
\begin{equation}\label{equi}
\widetilde{T}\circ \Phi_0=\Phi_0\circ T\,.
\end{equation}
Now fix $k$ and let $\mathcal{A}_k$ be the linear span of words as in \eqref{hamuni} for which $n=m$, $i_1,\ldots i_n,j_1,\ldots,j_n\leq k$, and $0\leq n\leq k$. Each $\mathcal{A}_k$ is then finite-dimensional, $\mathcal{A}_k\subset \mathcal{A}_{k+1}$, and for any $b\in\gb_0$, $T(b)\in\mathcal{A}_K$ with $K$ sufficiently large, and $\gb^\g=\overline{\bigcup_k \mathcal{A}_k}$. Moreover, since $\Phi_0$ is injective on $\gb_0$, it is isometric on each finite-dimensional subspace $\mathcal A_k\subset\gb_0$. \\
Finally, for any $b\in \gb_0$,
$$
\|b\|_{\max}=\lim_{n\to \infty} \|T((b^*b)^n)\|_{\max}^{1/2n}=
  \lim_{n\to \infty}\lim_{k\to \infty} \|T((b^*b)^n)\|_{\mathcal A_k}^{1/2n}\,.
$$
But\eqref{equi} gives that
$$
\|T((b^*b)^n)\|_{\mathcal A_k}=\|\Phi_0(T((b^*b)^n))\|
  =\|\tilde T(\Phi_0((b^*b)^n))\|\leq \|\Phi_0(b)\|^{2n}\,
$$
since $\tilde T$ is norm one. We then get $\|b\|_{\max}=\|\Phi_0(b)\|$ for all $b\in \gb_0$, and this extends to $\gb$ by density.
\end{proof}
We end the section by noting that $\gb$ is not a type $I$ $C^*$-algebra. In fact, by Proposition \ref{novacuum} and Proposition \ref{irre}, the Fock representation $\ga$ is irreducible but does not contain the compact operators, which contradicts a characterization of separable type $I$ $C^*$-algebras, see \cite{Gli}.

\section{The von Neumann algebra generated by position operators}\label{vna}
In this section we investigate the von Neumann algebra generated by the position operators associated with period-2 interacting Fock spaces. We prove that this algebra is a type $II_1$ factor. The key steps in the identification of the unique tracial state are the construction of a normal trace on the odd subalgebra and the fact that the vacuum vector is cyclic and separating for the even subalgebra.\\

The self-adjoint part of $\ga''$, namely the von Neumann algebra generated by $\{a_i+a^\dag_i \mid i\in\bn\}$, will be denoted by $\mathcal{S}$. The proof of Proposition \ref{irre} shows that $P_\Om$ belongs to $\ga''$. The next result shows that the identity operator is the unit of $\mathcal{S}$ and that both $P$ and $Q$ belong to it.

\begin{prop}
The von Neumann algebra $\mathcal{S}$ is unital with unit $I$ and contains both the projections $P$ and $Q$.
\end{prop}

\begin{proof}
Fix $i\in\bn$. From \eqref{crperiod}, we have
$$
s_i^2=a_ia_i+(\a P+ \b Q)+a^\dag_i a_i+a^\dag_ia^\dag_i,
$$
which shows that $\a P+ \b Q\in \mathcal{S}$ after taking the weak limit as $i\rightarrow\infty$.
As $\a\neq \b$, by functional calculus, $P, Q\in \mathcal{S}$ and hence $I=P+Q\in \mathcal{S}$.
\end{proof}
For each $n\in\bn$ and $i_1,\ldots, i_n\in\bn$, we introduce the normal (or Wick) ordered words
\begin{align}
\begin{split}\label{wick}
&W(i_1,\ldots, i_n):=\sum_{k=0}^n a^\dag_{i_1}\cdots a^\dag_{i_k}a_{i_{k+1}}\cdots a_{i_n}P,\\
&\widetilde{W}(i_1,\ldots, i_n):=\sum_{k=0}^n a^\dag_{i_1}\cdots a^\dag_{i_k}a_{i_{k+1}}\cdots a_{i_n}Q.
\end{split}
\end{align}
If $n=0$, we denote by $W(\emptyset)$ and $\widetilde{W}(\emptyset)$ the projections $P$ and $Q$, respectively. From now on, we denote by $\bf{W}$ the collection of all normal ordered words defined in \eqref{wick}.\\
We introduce the subalgebras $\mathcal{S}_P:=P\mathcal{S}P$ and $\mathcal{S}_Q:=Q\mathcal{S}Q$, and call them the even and odd parts of $\mathcal{S}$, respectively.

\begin{lem}\label{wickformula0}
The strong closure of the linear span of the words in \eqref{wick} coincides with $\mathcal{S}$. We also have
  \begin{eqnarray}\label{sub}
    \mathcal{S}_P=\overline{\{W(i_1,\ldots, i_{2n})\}}^{w*}=\langle
      W(i_1,i_2),P\rangle^{''}\,, \\\label{sub2} \mathcal{S}_Q=\overline{\{\widetilde W(i_1,\ldots, i_{2n})\}}^{w*}=\langle \widetilde{W}(i_1,i_2),Q\rangle^{''}\,,
\end{eqnarray}
where $n\in\bn$ and $i_1,\ldots,i_n\in\bn$.
\end{lem}
\begin{proof}
  The vector space generated by $\bf{W}$ is unital, since $I=W(\emptyset)+\widetilde{W}(\emptyset)$. It is invariant under the $*$-operation
  as $W(i_1,\ldots, i_n)^*=W(i_n,\ldots, i_1)$ if $n$ is even and
 $W(i_1,\ldots, i_n)^*=\widetilde W(i_n,\ldots, i_1)$ if $n$ is odd. Moreover, for each $j\in\bn$, we have $s_jP=W(j)$ and $s_jQ=\widetilde{W}(j)$. Therefore, it suffices to show that for each $j\in\bn$ and $i_1,\ldots, i_n\in\bn$, both $s_jW(i_1,\ldots, i_n)$ and $s_j\widetilde{W}(i_1,\ldots,i_n)$ belong to the linear span of $\bf{W}$.\\
To this end, introduce the functions, for $h\in\bn$,
$$
o(h):=\begin{cases}0 &\text{if }h=2m,\\
1&\text{if }h=2m+1,
\end{cases}\qquad
e(h):=\begin{cases}0 &\text{if }h=2m+1,\\
1&\text{if }h=2m,
\end{cases}
$$
and note that from \eqref{crperiod} we obtain
\begin{align}
\begin{split}
\label{swick1}
s_jW(i_1,\ldots, i_n)
&=(a_j+a^\dag_{j}) \sum_{k=0}^n a^\dag_{i_1}\cdots a^\dag_{i_k}a_{i_{k+1}}\cdots a_{i_n}P\\
&=W(j,i_1,\ldots,i_n)+\d_{j,i_1}(\a P+\b Q)W(i_2,\ldots,i_n)\\
&=W(j,i_1,\ldots,i_n)+\d_{j,i_1}\a^{o(n)}\b^{e(n)}W(i_2,\ldots,i_n),
\end{split}
\end{align}
and
\begin{align}
\begin{split}
\label{swick2}
s_j\widetilde{W}(i_1,\ldots, i_n)
&=(a_j+a^\dag_{j}) \sum_{k=0}^n a^\dag_{i_1}\cdots a^\dag_{i_k}a_{i_{k+1}}\cdots a_{i_n}Q\\
&=\widetilde{W}(j,i_1,\ldots,i_n)+\d_{j,i_1}(\a P+\b Q)\widetilde{W}(i_2,\ldots,i_n)\\
&=\widetilde{W}(j,i_1,\ldots,i_n)+\d_{j,i_1}\a^{e(n)}\b^{o(n)}\widetilde{W}(i_2,\ldots,i_n).
\end{split}
\end{align}
We now prove \eqref{sub}. For the first equality, one inclusion is
obvious as $W(i_1,i_2)\in \mathcal{S}_P$. Let us denote
$W_n:=\overline{\{W(i_1,\ldots, i_{2k});k\leq n\}}^{w*}$. Then, from
\eqref{swick1},
$s_{i_1}s_{i_2}W(i_3,...,i_{2n})-W(i_1,...,i_{2n})\in W_{n-1}$. In
particular, $s_{i_1}s_{i_2}P\in \langle W(i_1,i_2),P\rangle^{''}$ and one
concludes that
$W(i_1,...,i_{2n})\in \langle W(i_1,i_2),P\rangle^{''}$ by induction on $n$. Similarly, one gets \eqref{sub2}.
\end{proof}

\begin{rem}
The arguments in the above proof, together with Proposition \ref{basis}, show that the words in $\bf{W}$ form a Hamel basis for the $*$-algebra generated by the position operators.
\end{rem}

\begin{prop}\label{sep}
The vacuum vector $\Om$ is cyclic and separating for $\mathcal{S}_P$ on $\bigoplus_{n=0}^{\infty} \ch_{2n}$.
\end{prop}

\begin{proof}
First, observe that  for each $i_1,...,i_{2n}\in\bn$,
\[
W(i_1,...,i_{2n})\Om= e_{i_1}\otimes \cdots \otimes e_{2n},
\]
and \eqref{sub} ensures that $\Om$ is cyclic, since the vacuum vector is invariant under $P$. Therefore, to complete the proof, it suffices to show that $\Om$ is cyclic for the commutant $\mathcal{S}_P'$. To this end, we introduce the conjugation $J:\cf(\ch,\a,\b)\rightarrow \cf(\ch,\a,\b)$ defined by $J\Om=\Om$ and, for each $n\in\bn$ and $i_1,\ldots,i_n\in\bn$,
\[
J(e_{i_1} \otimes \cdots \otimes e_{i_{n}}):= e_{i_{n}}\otimes \cdots \otimes e_{i_1}.
\]
Notice that $J$ commutes with both $P$ and $Q$.\\
We claim that, for  $h,l,i,j\in\bn$, the operators $W(h,l)$
and $JW(i,j)J$  commute on $\bigoplus_{n=0}^{\infty} \ch_{2n}$.
By density, we need to check that
$$JW(i,j)JW(h,l) \xi=W(h,l)JW(i,j)J \xi,$$
for $\xi=e_{k_1}\otimes \cdots\otimes e_{k_{2n}}$, $n\geq 0$, $k_1,\ldots,k_{2n}\in \bn$.
To simplify the notation, we write $e_{k_1,\ldots,k_{2n}}$ instead of $e_{k_1}\otimes \cdots\otimes e_{k_{2n}}$.\\
For $n=0$, one has
\begin{align*}
JW(i,j)JW(h,l)\Omega&= J W(i,j) e_{l, h}    \\
&=\a\b \delta_{l,j}\delta_{i,h} \Omega+ \b \delta_{l,j}
e_{h,i}+ e_{h,l,j,i},
\end{align*}
and
\begin{align*}
W(h,l)JW(i,j)J\Omega&=  W(h,l) e_{j, i}\\
&=\a\b\delta_{l,j}\delta_{h,a} \Omega+ \b \delta_{b,j}
e_{h,i}+ e_{h,l,j,i}\,.
\end{align*}
Similarly, for $n=1$, using the periodicity of $(\om_n)_n$,
\begin{align*}
A&:=JW(i,j)JW(h,l)e_{k_1,k_2}\\
&= J W(i,j)
    \big(\delta_{k_1,l}\delta_{k_2,h}\a\b \Omega +
    \delta_{k_1,l}\b e_{k_2,h}+e_{k_2,k_1,l,h})\,.
\end{align*}
Thus,
\begin{align*}
%\begin{split}
A=&\delta_{k_1,l}\delta_{k_2,h}\a\b e_{j,i}+
\delta_{k_1,l}\delta_{k_2,j}\delta_{i,h}\a\b\Omega+
\delta_{k_1,h}\delta_{k_2,j}\b^2e_{h,i}\\
&+\delta_{k_1,l}\b e_{h,k_2,j,i}+ \delta_{j,k_2}\delta_{i,k_1}\a\b
    e_{h,l}+ \delta_{k_2,j}\b e_{h,l,k_1,i} +
      e_{h,l,k_1,k_2,j,i}.
\end{align*}
The other way round,
\begin{align*}
B&:=W(h,l)JW(i,j)Je_{k_1,k_2}\\
&=  W(h,l)
    (\delta_{k_2,j}\delta_{k_1,i}\a\b \Omega +
    \delta_{k_2,j}\b e_{k_1,i}+e_{k_1,k_2,j,i})\,.
\end{align*}
Since
\begin{align*}
B=&\delta_{k_2,j}\delta_{k_1,i}\a\b e_{h,l}+
\delta_{k_2,j}\delta_{l,k_1}\delta_{h,i}\a\b\Omega+
       \delta_{k_2,j} \delta_{l,k_1}\b^2e_{h,i}\\
       &+\delta_{k_2,j}\b e_{h,l,k_1,i}+
\delta_{k_1,l}\delta_{k_2,h}\a\b e_{j,i}+\delta_{k_1,l}
       \b e_{h,k_2,j,i}+ e_{h,l,k_1,k_2,j,i}\,,
\end{align*}
one gets that $A=B$.
   The computation for $n\geq 2$ is easier, since the operators
   $W(h,l)$ act only on the first two tensors $e_{k_1}\otimes e_{k_2}$ whereas
$JW(i,j)J$ on the last two $e_{k_{n-1}}\otimes e_{k_n}$. We conclude thanks to the periodicity of $(\omega_l)_l$.\\
By \eqref{sub}, this commutation implies that $J\mathcal{S}_PJ\subseteq \mathcal{S}'_P$. Since ${S}'_P\Om=J\mathcal{S}_P\Om$, the cyclicity of $\Om$ for $\mathcal{S}'_P$ follows from that of $\Om$ for $\mathcal{S}_P$.
\end{proof}
\begin{rem} As $Q \Omega=0$, the vacuum vector is not separating for $\mathcal S$.
\end{rem}
\begin{prop}\label{evenfactor}
$\mathcal{S}_P$ is a type $II_1$ factor, with trace given by the vacuum expectation.
\end{prop}

\begin{proof}
We prove that the vacuum state $\t_{\Om}$ is the unique faithful normal trace.\\
We first prove the tracial property. By \eqref{sub}, it suffices to verify that
$\t_{\Om}(W(i_1,i_2)x)=\t_{\Om}(xW(i_1,i_2))$
for each $i_1,i_2\in\bn$ and $x\in\mathcal{S}_P$. The unit of $\mathcal{S}_P$ is $P$, and  Proposition \ref{sep} shows that
$J\mathcal{S}_PJ\subset \mathcal{S}'_P$. In addition, $W(i_1,i_2)^*=W(i_2,i_1)$, and therefore
\begin{align*}
\langle W(i_1,i_2)x\Om,\Om\rangle
&= \langle x\Om, e_{i_2}\otimes e_{i_1}\rangle\\
&= \langle x\Om, JW(i_1,i_2)J\Om\rangle\\
&= \langle JW(i_2,i_1)Jx\Om,\Om\rangle\\
&= \langle xJW(i_2,i_1)J\Om,\Om\rangle\\
&= \langle x W(i_1,i_2)\Om,\Om\rangle.
\end{align*}
To prove uniqueness, suppose that $\t$ is another normal trace on $\mathcal{S}_P$, that is,
\[
\t(x):=\sum_{i=1}^{\infty} \langle x\xi_i,\xi_i\rangle,
\quad \xi_i\in P\cf(\ch,\a,\b),\quad \sum_{i=1}^{\infty}\|\xi_i\|^2=1\,, \quad x\in\mathcal{S}_P\,.
\]
Let us take $j\in\bn$.  We have $W(j,j)^2=\a\b P + y_j$, where $(y_j)$ is a bounded family converging weakly to zero as $j\to\infty$, we deduce
\[
\t(x)=\frac{1}{\a\b}\lim_{j\rightarrow \infty} \t(xW(j,j)^2).
\]
Fix $x:= W(i_1,\ldots, i_n)$, where $i_1,\ldots,i_n\in\bn$. Then
\[
\frac{1}{\a\b} \lim_{j\rightarrow \infty}\langle xW(j,j)\xi_i,W(j,j)\xi_i\rangle
=\begin{cases}
0 &\text{if } n>0,\\
\|\xi_i\|^2 &\text{if } n=0.
\end{cases}\,.
\]
By \eqref{norm}, there exists a positive constant $C_{\a,\b}$, depending only on $\a$ and $\b$, such that $|\langle xW(j,j)\xi_i,W(j,j)\xi_i\rangle|\leq C_{\alpha,\beta} \|x\|^2\|\xi_i\|^2$ for each $j$. By the Lebesgue dominated convergence theorem, we conclude that
\[
\t(x)=\d_{n,0}=\langle x\Om,\Om\rangle=\t_{\Om}(x).
\]
Using \eqref{sub}, the result follows by a standard approximation argument.
\end{proof}

The factoriality of $\mathcal{S}_Q$ requires additional preliminary results which allow us to overcome the delicate problem to extend the state $\t_o$ defined on the $*$-algebra by $\tau_o(\widetilde W(i_1,....i_{2n}))=\delta_{n,0}$ to the unique normal trace on the whole $\mathcal{S}_Q$. We begin by describing how the kernel changes according to the order relation between $\a$ and $\b$.

\begin{lem}\label{kernel}
For each $i\in\bn$, the kernel of $s_i$ consists of the closed span of vectors of the form
\begin{equation*}%\label{ker}
x:=\sum_{k=0}^\infty {(-1)^k}{\g^{k}}e_i^{\otimes 2k}\otimes v,
\end{equation*}
where $v=e_{j_1}\otimes \cdots \otimes e_{j_l}$ with $l\in\bn$, $j_1\neq i$, and
$\g=\sqrt{\frac \b \a}$ if $l$ is odd and $\frac \b \a<1$, or $\g=\sqrt{\frac \a \b}$ if $l$ is even and $\frac \a \b<1$.
\end{lem}
\begin{proof}
Take an arbitrary $i\in\bn$ and fix
  $v=e_{j_1} \otimes \cdots \otimes e_{j_l}$ with
  $l\in\bn$ ($v=\Omega$ if $l=0$), and $j_1\neq i$. We introduce the vector
$$
f_k:=\frac{e_i^{\otimes k}\otimes v} {\| e_i^{k}\otimes v\|}= \frac{e_i^{\otimes k}\otimes
  v}{\sqrt{\lambda_{k+l}}}\,, \quad k\geq 0\,.
$$
After recalling that  for each $n$, $\omega_n= \frac {\lambda_n}{\lambda_{n-1}}$, we can compute (with
  the convention $f_{-1}=0$)
$$
s_if_k= \sqrt{\frac{\lambda_{k+l+1}}{\lambda_{k+l}}}f_{k+1}+
  \omega_{k+l} \sqrt{\frac{\lambda_{k+l-1}}{\lambda_{k+l}}}f_{k-1}=
  \sqrt{\omega_{k+l+1}} f_{k+1}+ \sqrt{\omega_{k+l}}f_{k-1}\,.
$$
In particular, the span of $(f_k)_k$ is invariant under $s_i$, and we may determine the kernel of $s_i$ within this invariant subspace. Assume $s_i\big(\sum_{k=0}^\infty \l_k f_k\big)=0$, where $\l_k\in\bc$ for each $k$. Testing with $f_j$, $j\geq 0$, we must have that
$$
\sqrt{\omega_{j+l+1}} x_{j+1}+ \sqrt{\omega_{j+l}}x_{j-1}=0\,.
$$
In particular, $x_{2j+1}=0$ and
  $x_{2j+2}=
 - \sqrt{\frac{\omega_{2j+l+1}}{\omega_{2j+l+2}}}x_{2j}=-\sqrt{\frac{\omega_{l+1}}{\omega_{l}}}
  x_{2j}$.  This yields the result depending on the parity of $l$.
\end{proof}
Since the position operator $s_i$, $i\in\bn$, is self-adjoint, its polar decomposition is given by a partial isometry (symmetry) $u$ on $(\ker s_i)^{\perp}=\overline{R(s_i)}$ and vanishes on $\ker s_i$, so that $s_i=u|s_i|$. \\
If we call the vectors in the range of $P$ even, and odd those in the range of $Q$, we say that an operator in $\mathcal S$ is even if it preserves parity and we say it is odd if it maps even vectors to odd ones and conversely. Clearly, $s_i$ is odd and $|s_i|$ is even by functional calculus. Hence,
$u=\mathrm{w}\!-\!\lim_{\eps\rightarrow 0} s_i(|s_i|+\eps I)^{-1}$ is odd.\\
If $\a<\b$, then the kernel of $s_i$ sits in the range of $P$, thus the range of $s_i$ contains $Q$. It follows that $v=Qu$, satisfies $vv^*=Q$ and $v^*v=uQu\leq  P$ because $u$ is odd. We have proved:
\begin{lem}\label{odd}
  If $\a<\b$, there exists a partial unitary $v\in\mathcal{S}$ with $vv^*=Q$ and
  $v^*v\leq P$.
\end{lem}
In the other case, we will need one more computation, which has appeared in other contexts \cite{Sn}.
\begin{lem} \label{inverse}
If $\a\neq\b$, there exists some $n\geq 1$ such that $\sum_{i=1}^n Qs_i^2Q$ is invertible in $\mathcal S_Q$.
\end{lem}
\begin{proof}
By \eqref{crperiod}, one obtains
  $$\sum_{i=1}^nQs_i^2Q=\sum_{i=1}^na_ia_iQ +\Big(n \b Q +\sum_{i=1}^n a^\dag_i a_iQ\Big)+\sum_{i=1}^n a^\dag_ia^\dag_iQ=S_n+I_n+S_n^*\,.
  $$
  Clearly, $I_n\geq n\beta Q$ and is invertible on $\mathcal{S}_Q$. Therefore, $\displaystyle{\|I_n^{-1}\|\leq \frac 1{n\b}}$. Since the $a_i^\dagger a_i^\dagger Q$ have orthogonal ranges and all have the same norm, $\displaystyle{\|S_n\|\leq C_{\a,\b} \sqrt n}$, where $C_{\a,\b}:=\max\{\sqrt{\a},\sqrt{\b}\}$. Thus, by the Neumann Lemma, we have the conclusion if
  $\displaystyle{2C_{\a,b}<\b \sqrt n}$.
\end{proof}
\begin{lem}\label{even}
  If $\a>\b$, there exist partial unitaries $(v_i)_{i=1}^n$ in $\mathcal S$
  with $\sum_{i=1}^n v_iv_i^*=Q$ and  $v_i^*v_i\leq P$ for all $i=1,...n$.
\end{lem}
\begin{proof}
  Let $n$ be given as in the previous lemma and $s_i=u_i|s_i|$ be the
  polar decompositions of $s_i$. Each $u_i$ is a partial isometry  and, since $\a>\b$, by
  Lemma \ref{kernel} we have $u_i^2=I-\gamma_i$, where $\gamma_i\leq Q$ is the projection on the kernel of $s_i$. As $u_i$ is odd, this forces $u_iPu_i=Q-\gamma_i$ and $u_iQu_i=P$.\\
If we set $v_1:=Qu_1=u_1P$, then $q_1:=v_1v_1^*=Q-\g_1$ and $v_1^*v_1=P$.\\
Assume by induction that we are given the partial unitaries $v_i$, $i=1,\ldots,k$ ($k<n$), such that
  $q_k:=\sum_{i=1}^{k} v_iv_i^*\leq Q$, $v_i^*v_i\leq P$, and finally $Q-q_k\leq \g_j$ for each $j=1,\ldots, k$.\\
  Consider the odd operator $(Q-q_k)s_{k+1}$, with its
  left and right support, denoted by $q$ and $p$, respectively. Its polar decomposition gives rise to the partial unitary $v_{k+1}\in \mathcal S$, where
   $q=v_{k+1}v_{k+1}^*\leq Q-q_k$ and $p=v_{k+1}^*v_{k+1}\leq P$. Denote $q_{k+1}:=q_k+q$, and pick $\eta$ in the range of $Q-q_{k+1}$, which, by the assumption above, is dominated by each $\gamma_j$, $j=1,\ldots,n$. Then $(Q-q_k)\eta=\eta$ and $q \eta=0$. As by definition $q(Q-q_k)s_{k+1}=(Q-q_k)s_{k+1}$, for each $\xi\in\cf(\ch;\a,\b)$, one has
   $$
   \langle s_{k+1}\eta, \xi\rangle= \langle \eta, (Q-q_k)s_{k+1}\xi\rangle=\langle \eta, q(Q-q_k)s_{k+1}\xi\rangle=0\,.
   $$
   This means that $\eta$ is in the kernel of $s_{k+1}$, and therefore $Q-q_{k+1}\leq \g_{k+1}$. Thus, the induction step is complete.\\
Since $Q-q_n\leq \g_j$ for each $j=1,\ldots,n$, if   $\eta$ is in the range of $Q-q_{n}$, we have $\g_j\eta=\eta$ and then $s_j \eta=0$. It follows that $\sum_{i=1}^n s_i^2 \eta=0$ and thus $\eta=0$ by Lemma \ref{inverse}. This gives that $q_n=Q$.
\end{proof}
\begin{prop}\label{sqfact}
  The von Neumann algebra $\mathcal{S}_Q$ has a unique tracial state
  $\tau_o$ given by $\tau_o(\widetilde W(i_1,....i_{2n}))=\delta_{n,0}$ for $i_1,...,i_n\in \bn$ and $n\geq 0$.
\end{prop}
\begin{proof}
 The argument of Proposition \ref{evenfactor} to show the uniqueness of the trace also works for $\mathcal S_Q$ replacing $W$ by $\widetilde W$. In particular,
 the only possibility is that  $\tau_0(\widetilde W(i_1,\ldots,i_{2n}))=\delta_{n,0}$ for all $i_1,\ldots,i_{2n}\in\bn$. This yields the uniqueness by weak-$*$ density \eqref{sub2}.
 We now show that $\mathcal S_Q$ indeed admits a normal trace.\\
When $\a<\b$, by Lemma \ref{odd}, there exists a partial isometry
  $v\in\mathcal{S}$ such that $v^*v\leq P$ and $Q=vv^*$. Define
  $\pi:\mathcal{S}_Q\rightarrow \mathcal{S}_P$ by
  $\pi(x):=v^*xv$. Since $vv^*$ is the identity of $\mathcal{S}_Q$,
  the map $\pi$ is an injective $*$-homomorphism. Then
  $\tau_\Omega\circ \pi$ is a finite non zero normal positive functional on $\mathcal{S}_Q$.\\
When $\b>\a$, by Lemma \ref{even} there exist partial isometries
  $v_i\in\mathcal{S}$, $i=1,\ldots,n$, such that $ v_i^*v_i\leq P$ and
  $\sum_{i=1}^n v_iv_i^*=Q$.  Define, for $x\in\mathcal S_Q$,
  $\tau(x):=\sum_{i=1}^n \tau_\Omega(v_i^*xv_i)$. This is a normal positive functional on $\mathcal S_Q$. For $x,y\in\mathcal S_Q$, as $v_i^*xv_j, v_j^*yv_i\in \mathcal S_P$, one has
  $$\tau(xy)=\sum_{i,j=1}^n \tau_\Omega(v_i^*xv_jv_j^*yv_i)
  =\sum_{i,j=1}^n \tau_\Omega(v_j^*yv_iv_i^*xv_j)=\tau(yx).$$
  This concludes the proof.
\end{proof}

\begin{thm}\label{factorr}
The von Neumann algebra $\mathcal{S}$ is a
  factor of type $\text{II}_1$.
  \end{thm}

  \begin{proof}
We first observe that for each $j, m \in \mathbb{N}$ and $i_1, \ldots, i_m \in \mathbb{N}$,
\begin{align}
\begin{split}\label{ws}
W(i_1, \ldots, i_m) s_j
&= \sum_{k=0}^{m} a^\dagger_{i_1} \cdots a^\dagger_{k} a_{k+1} \cdots a_{m} P(a_j + a^\dagger_j) \\
&= \widetilde{W}(i_1, \ldots, i_m, j) + \delta_{i_m, j} \beta \widetilde{W}(i_1, \ldots, i_{m-1}),
\end{split}
\end{align}
and
\begin{align}
\begin{split}\label{tws}
\widetilde{W}(i_1, \ldots, i_m) s_j
&= \sum_{k=0}^{m} a^\dagger_{i_1} \cdots a^\dagger_{k} a_{k+1} \cdots a_{m} Q(a_j + a^\dagger_j) \\
&= W(i_1, \ldots, i_m, j) + \delta_{i_m, j} \alpha W(i_1, \ldots, i_{m-1}).
\end{split}
\end{align}
Let us define the normal state $\tau: \mathcal{S} \to \mathbb{C}$ by
\begin{equation}\label{trac}
\tau(x) := \frac{1}{\alpha + \beta} \left( \beta \tau_\Omega(PxP) + \alpha \tau_0(QxQ) \right), \quad x \in \mathcal{S},
\end{equation}
and show that it is tracial. To this end, it suffices to prove that for any $j, m \in \mathbb{N}$ and $i_1, \ldots, i_m \in \mathbb{N}$,
\begin{equation}\label{trace1}
\tau(W(i_1, \ldots, i_m) s_j) = \tau(s_j W(i_1, \ldots, i_m)),
\end{equation}
and
\begin{equation}\label{trace2}
\tau(\widetilde{W}(i_1, \ldots, i_m) s_j) = \tau(s_j \widetilde{W}(i_1, \ldots, i_m)).
\end{equation}
If $m=2n$, then from \eqref{ws}, \eqref{swick1}, \eqref{swick2} and \eqref{tws}, both the left-hand side and the right-hand side of equations \eqref{trace1} and \eqref{trace2} vanish.\\
If $m=2n+1$, then \eqref{ws} and \eqref{swick1} yield
\begin{align*}
\tau(W(i_1, \ldots, i_{2n+1}) s_j)
&= \frac{\alpha \beta}{\alpha + \beta} \delta_{i_{2n+1}, j} \tau_0(Q \widetilde{W}(i_1, \ldots, i_{2n})) \\
&= \frac{\alpha \beta}{\alpha + \beta} \delta_{n,0} \delta_{i_1, j} \\
&= \frac{\alpha \beta}{\alpha + \beta} \delta_{i_1, j} \tau_\Omega(P W(i_2, \ldots, i_{2n})) \\
&= \tau(s_j W(i_1, \ldots, i_{2n+1})).
\end{align*}
Similarly, \eqref{tws} and \eqref{swick2} give
\begin{align*}
\tau(\widetilde{W}(i_1, \ldots, i_{2n+1}) s_j)
&= \frac{\alpha \beta}{\alpha + \beta} \delta_{i_{2n+1}, j} \tau_\Omega(P W(i_1, \ldots, i_{2n})) \\
&= \frac{\alpha \beta}{\alpha + \beta} \delta_{n,0} \delta_{i_1, j} \\
&= \frac{\alpha \beta}{\alpha + \beta} \delta_{i_1, j} \tau_0(Q \widetilde{W}(i_2, \ldots, i_{2n})) \\
&= \tau(s_j \widetilde{W}(i_1, \ldots, i_{2n+1})).
\end{align*}
Finally, the trace $\tau$ is unique. Indeed, any normal tracial state on $\mathcal{S}$ must be a convex combination of $\tau_\Omega$ and $\tau_o$. That is, for each $x \in \mathcal{S}$,
\begin{equation*}%\label{eqtr}
\tilde{\tau}(x) = \tilde{\tau}(PxP) + \tilde{\tau}(QxQ) = \lambda \tau_\Omega(PxP) + (1 - \lambda) \tau_0(QxQ),
\end{equation*}
where $\lambda \in [0, 1]$.
Relations \eqref{swick1} and \eqref{ws} force $\lambda$ to satisfy $\alpha \lambda = (1 - \lambda) \beta$, and therefore $\tilde{\tau} = \tau$.
\end{proof}
\section{free Poisson laws and free group factors}\label{seclaw}
This section is devoted to the probabilistic aspects of period-2 interacting Fock spaces. We compute explicitly the vacuum distribution of the position operators and show that it is related to free Poisson law through a square-root transformation. We then prove that the squares of the position operators are free with respect to the vacuum state, yielding copies of the free group factor $L(\mathbb F_\infty)$ inside the generated von Neumann algebra.\\

We introduce the notation
\begin{equation}\label{pn}
P_n(\a,\b):=\sum_{\eps\in \{-1,1\}^{2n}_+}
\langle\Om,a_j^{\eps(1)}\cdots a_j^{\eps(2n)}\Om\rangle,
\qquad n\in\bn^*,
\end{equation}
where the independence of the left-hand side from $j$ follows from Proposition~\ref{moment1}.
A direct consequence of Proposition~\ref{moment1} is the following result.
\begin{prop}\label{moment2}
For each $n\in\bn^*$ and $\eps\in\{-1,1\}^{2n}_+$, there exists $k\in\{0,1,\ldots,n\}$, determined by $\eps$, such that for each $j\in\bn$,
\begin{align*}%\label{2ome02h}
\t_\Om(a_j^{\eps(1)}\cdots a_j^{\eps(2n)})
=\a^k\b^{n-k}, \qquad
\t_\Om(a_ja_j^{\eps(1)}\cdots a_j^{\eps(2n)}a^\dag_j)
=\a\b^k\a^{n-k}.
\end{align*}
\end{prop}
\begin{proof}
As already shown, the indices $l_1,\ldots,l_n$ appearing in \eqref{2ome02a} and \eqref{2ome02b} are uniquely determined by $\eps\in\{-1,1\}^{2n}_+$. The claim follows by observing that $\om_{2h-l_k}$ equals $\a$ or $\b$ depending on whether $l_k$ is odd or even, respectively.
\end{proof}
Recall that for each $\eps\in \{-1,1\}^{2n}_+$, there is a unique $k\in\{0,1,\ldots,n\}$ such that for any $j\in\bn$
\[
\langle\Om,a_j^{\eps(1)} \ldots a_j^{\eps(2n)}\Om \rangle=\a^k\b^{n-k}\,.
\]
If for any $k=0,\ldots,n$, we denote
\[
J_{n,k}:=\big\{\eps\in \{-1,1\}^{2n}_+:\,\langle\Om, a_j^{\eps(1)}\ldots a_j^{\eps(2n)}\Om \rangle =\a^k\b^{n-k}\big\},
\]
then one has $\{-1,1\}^{2n}_+=\bigsqcup_{k=0}^n J_{n,k}$, where the symbol $\bigsqcup$ stands for disjoint union.  Therefore, for $p_{n,k}:= |J_{n,k}|$,
\begin{equation}\label{cata}
\sum_{k=0}^n p_{n,k} = \bigg|\bigsqcup_{k=0}^n J_{n,k}\bigg| = |\{-1,1\}^{2n}_+| = \frac{1}{n+1}\binom{2n}{n},
\end{equation}
where the last equality follows where the last equality follows from the well-known fact that the number of non crossing pair partitions of $[2n]$ is the $n$th Catalan number.\\
In addition, \eqref{pn} gives us that for each $n\in\bn^*$
\[
P_n(\a,\b)=\sum_{k=0}^n p_{n,k}\a^k\b^{n-k},
\]
and
\begin{equation}\label{acata}
\sum_{\eps\in \{-1,1\}^{2n}_+} \t_\Om(a_j a_j^{\eps(1)} \ldots a_j^{\eps(2n)} a_j^\dag)= \a \sum_{k=0}^n p_{n,k} \b^k \a^{n-k} = \a P_n(\b,\a).
\end{equation}
Let us define the (formal) power series $S_{u,v}(t):=\sum_{n=0}^{\infty} P_n(u,v) t^n$, where $P_0 := 1$, $u,v\in\bc$, $t\in\bc$, and notice that
\begin{align*}
&\a + \sum_{n=1}^\infty \sum_{\eps\in \{-1,1\}^{2n}_+} \langle \Om, a_j a_j^{\eps(1)} \cdots a_j^{\eps(2n)} a_j^\dag \Om \rangle t^n \\
=& \a + \a \sum_{n=1}^\infty P_n(\b,\a) t^n = \a S_{\b,\a}(t).
\end{align*}
For each $j\in\bn$, the moment generating function of $s_j^2$ with respect to the vacuum will be denoted by $\cam_{s^2}$.

\begin{prop}\label{2ome03}
The series $S_{u,v}(t)$ has a positive radius of convergence. Namely, there exists $\delta>0$ such that the series converges for all $t\in\bc$ with $|t|<\delta$. Therefore, in the same disc, $\cam_{s^2}=S_{\a,\b}$.
\end{prop}

\begin{proof}
Indeed, for any $n\in\bn$ and $u,v\in\bc$, \eqref{cata} entails
\[
|P_n(u,v)| \le (n+1) \max_{0\le k\le n} p_{n,k} \big(\max\{|u|,|v|\}\big)^n \le \binom{2n}{n} \big(\max\{|u|,|v|\}\big)^n.
\]
The last part of the statement then follows from \eqref{2ome01b} and \eqref{pn}.
\end{proof}

Note that the map
\[
z \mapsto \sum_{n=0}^\infty \sum_{\eps\in \{-1,1\}^{2n}_+} \t_\Om(a_j a_j^{\eps(1)} \cdots a_j^{\eps(2n)} a_j^\dag)  z^n
\]
is also well-defined for $|z|$ sufficiently small, and the series sums to $\a S_{\a,\b}(z)$, as follows from \eqref{acata}.

\begin{thm}\label{MarPa}
The moment generating function $\cam_{s^2}$, within its radius of convergence, takes the form
\begin{equation}\label{MP}
\cam_{s^2}(z) = S_{\a,\b}(z) = \frac{2}{1 - (\a-\b)z + \sqrt{\big(1 - (\a-\b)z\big)^2 - 4\b z}}.
\end{equation}
Therefore, for any $j\in\bn$, $s_j^2$ has the Marchenko-Pastur distribution with rate $\frac{\a}{\b}$ and size $\b$ with respect to the vacuum.
\end{thm}
\begin{proof}
  Let us fix $j\in\bn$. For each $k\in\{1,\ldots,n\}$, let $\{-1,1\}^{2n}_{+,k}$ be the collection of all $\eps\in\{-1,1\}^{2n}_{+}$ for which the corresponding non crossing pair partition $\{(l^\eps_h,r^\eps_h)\}_{h=1}^n$ satisfies {$(l^\eps_1,r^\eps_1)=(1,2k)$}.\\
Clearly,
\[
\bigsqcup_{k=1}^n \{-1,1\}^{2n}_{+,k} = \{-1,1\}^{2n}_+.
\]
Moreover, for each $k\in\{1,\ldots,n\}$ and $\eps\in\{-1,1\}^{2n}_{+,k}$, one easily obtains the well-known factorization
\begin{align}\label{facto}
\t_\Om(a_j^{\eps(1)} \cdots a_j^{\eps(2n)})
&= \t_\Om(a_j^{\eps(1)} \cdots a_j^{\eps(2k)})
  \t_\Om(a_j^{\eps(2k+1)} \cdots a_j^{\eps(2n)}) \\
&= \t_\Om(a_j a_j^{\eps(2)} \cdots a_j^{\eps(2k-1)} a_j^\dag) \t_\Om (a_j^{\eps(2k+1)} \cdots a_j^{\eps(2n)}).
\end{align}
This factorization suggests introducing two sub non crossing pair partitions derived from $\eps\in\{-1,1\}^{2n}_{+,k}$, with $k\in\{1,\ldots,n\}$ as usual. Namely, we define
\begin{align*}
\eps_1(h) &:= \eps(h+1), \quad h\in\{1,\ldots,2(k-1)\}, \\
\eps_2(h) &:= \eps(2k+h), \quad h\in\{1,\ldots,2(n-k)\},
\end{align*}
and notice that $\eps_1 \in \{-1,1\}^{2(k-1)}_+$ and $\eps_2 \in \{-1,1\}^{2(n-k)}_+$.\\
Now we compute, for $|z|<\delta$ (it is the same for $S_{\a,\b}$ and $S_{\a,\b}$)
\begin{align*}
S_{\a,\b}(z)
&= \sum_{n=0}^\infty \sum_{\eps\in \{-1,1\}^{2n}_+} \t_\Om( a_j^{\eps(1)} \cdots a_j^{\eps(2n)}) z^n \\
&= 1+ \sum_{n=1}^\infty \sum_{k=1}^n \sum_{\eps\in \{-1,1\}^{2n}_{+,k}} \t_\Om( a_j^{\eps(1)} \cdots a_j^{\eps(2n)}) z^n.
\end{align*}
Using the factorization property, the last line becomes
\begin{align*}
1 + \sum_{n=1}^\infty \sum_{k=1}^n
\sum_{\eps_1\in\{-1,1\}^{2(k-1)}_+}
\sum_{\eps_2\in\{-1,1\}^{2(n-k)}_+}
& \t_\Om(a_j a_j^{\eps_1(1)} \cdots a_j^{\eps_1(2(k-1))} a_j^\dag) \cdot \\
& \cdot \t_\Om( a_j^{\eps_2(1)} \cdots a_j^{\eps_2(2(n-k))}) z^n.
\end{align*}
Rearranging indices in the sums, we get
\begin{align*}
1 + &\sum_{k=1}^\infty \sum_{\eps_1\in\{-1,1\}^{2(k-1)}_+}
\t_\Om(a_j a_j^{\eps_1(1)} \cdots a_j^{\eps_1(2(k-1))} a_j^\dag) z^k
\cdot\\
& \cdot \sum_{m=0}^\infty \sum_{\eps_2\in\{-1,1\}^{2m}_+}
\t_\Om(a_j^{\eps_2(1)} \cdots a_j^{\eps_2(2m)}) z^m.
\end{align*}
Since $\sum_{m=0}^\infty \sum_{\eps_2\in\{-1,1\}^{2m}_+} \t_\Om( a_j^{\eps_2(1)} \cdots a_j^{\eps_2(2m)}) z^m = S_{\a,\b}(z)$, it follows that
\begin{align}\label{2ome04b}
\begin{split}
S_{\a,\b}(z)
&= 1 + S_{\a,\b}(z) \sum_{k=1}^\infty \sum_{\eps_1\in\{-1,1\}^{2(k-1)}_+} \t_\Om(a_j a_j^{\eps_1(1)} \cdots a_j^{\eps_1(2(k-1))} a_j^\dag) z^k \\
&= 1 + z S_{\a,\b}(z) \sum_{m=0}^\infty \sum_{\eps\in\{-1,1\}^{2m}_+} \t_\Om(a_j a_j^{\eps(1)} \cdots a_j^{\eps(2m)} a_j^\dag) z^m \\
&= 1 + z S_{\a,\b}(z) \Big( \t_\Om(a_j a_j^\dag)+ \sum_{m=1}^\infty \sum_{\eps\in\{-1,1\}^{2m}_+} \t_\Om(a_j a_j^{\eps(1)} \cdots a_j^{\eps(2m)} a_j^\dag) z^m \Big) \\
&= 1 + \a\, z\, S_{\a,\b}(z) \, S_{\b,\a}(z).
\end{split}
\end{align}
Furthermore, this equality holds for any $\a, \b > 0$, so that
\begin{align}\label{2ome04c}
S_{\b,\a}(z) = 1 + \b\, z\, S_{\a,\b}(z) \, S_{\b,\a}(z).
\end{align}
Applying \eqref{2ome04c} repeatedly in \eqref{2ome04b} yields
\begin{align*}
S_{\a,\b}(z) &= 1 + \a\, z\, S_{\a,\b}(z) \, S_{\b,\a}(z) \\
&= 1 + \a\, z\, S_{\a,\b}(z) \big(1 + \b\, z\, S_{\a,\b}(z) \, S_{\b,\a}(z) \big) \\
&= 1 + \a\, z\, S_{\a,\b}(z) + \a \b z^2 S_{\a,\b}(z)^2 S_{\b,\a}(z) \\
&= \dots = 1 + \frac{\a\, z\, S_{\a,\b}(z)}{1 - \b\, z\, S_{\a,\b}(z)}.
\end{align*}
Consequently, the quadratic equation
\[
\b\, z\, S_{\a,\b}^2 - \big(1 - (\a-\b) z\big) S_{\a,\b} + 1 = 0
\]
holds. Using the initial condition $S_{\a,\b}(0) = 1$ gives \eqref{MP}. Finally, exploiting \eqref{AA}, a simple algebraic manipulation leads to \eqref{CaMP} when $\g = \b$ and $\l = \frac{\a}{\b}$.
\end{proof}

\begin{cor}\label{norm}
For each $j\in\bn$, one has $\|s_j\| = \sqrt{\a} + \sqrt{\b}$.
\end{cor}
\begin{proof}
Fix $j\in\bn$. Notice that $P s_j^2 P \in \mathcal{S}_P$ and is self-adjoint. By Proposition \ref{sep}, the vacuum distribution of $P s_j^2 P$ is supported on the entire spectrum. Moreover, for each $n\in\bn$, since $P s_j^2 = s_j^2 P$,
\[
\t_\Om((P s_j^2 P)^n) = \t_\Om(P s_j^{2n} P) = \langle (s_j^2)^n \Om, \Om \rangle.
\]
Since the moments agree, the vacuum distribution of $P s_j^2 P$ is the Marchenko-Pastur law given in Theorem \ref{MarPa}. Now $s_j^2$ and $P$ are commuting self-adjoint operators, so $\|s_j^2 P\| = \|s_j^2\| = \|s_j\|^2$. The claim then follows from Theorem \ref{MarPa}, as $s_j^2 P = s_j^2 P^2 = P s_j^2 P$.
\end{proof}
Lemma \ref{kernel} implies that the distribution of $s_i$ has an atom at $0$ whenever $\a<\b$. We make this statement more precise below.
\begin{prop}\label{distri}
For each $j\in\bn$, $s_j$ has the vacuum distribution
\[
\n :=
\begin{cases}
(1 - \frac{\a}{\b}) \delta_0 + \widetilde{\nu} & \text{if $\a \le \b$,} \\
\widetilde{\nu} & \text{if $\a > \b$,}
\end{cases}
\]
where
\[
\di \widetilde{\nu}(x) := \frac{1}{2\pi \b |x|} \sqrt{2(\alpha+\beta)x^2 - (\alpha-\beta)^2 - x^4}\, \chi_{(-\gamma, -\eta) \cup (\eta, \gamma)}(x) \di x,
\]
and $\eta := |\sqrt{\alpha} - \sqrt{\b}|$, $\gamma := \sqrt{\alpha} + \sqrt{\b}$.
\end{prop}

\begin{proof}
Let $\cam_s$ be the moment generating function of any $s_j$, $j\in\bn$, and notice that for each $z\in\bc$,
\[
\cam_s(z) = S_{\a,\b}(z^2) = \frac{1 - (\alpha-\beta) z^2 - \sqrt{(1-(\alpha-\beta) z^2)^2 - 4 \beta z^2}}{2 \beta z^2}.
\]
The Cauchy transform is
\begin{align}
\begin{split}
\label{CT1}
\mathcal{G}_{\a,\b}(z)
&= \frac{1}{z} M_{\a,\b}\Big(\frac{1}{z}\Big)
= \frac{1 - \frac{\alpha-\beta}{z^2} - \sqrt{\big(1 - \frac{\alpha-\beta}{z^2}\big)^2 - \frac{4 \beta}{z^2}}}{2 \beta / z} \\
&= \frac{1}{2 \beta z} \big( z^2 - (\alpha-\beta) - \sqrt{z^4 - 2(\alpha+\beta) z^2 + (\alpha-\beta)^2} \big).
\end{split}
\end{align}
Thus, $z=0$ is a simple pole. Since
\[
\text{Res}_{z=0} \mathcal{G}_{\a,\b}(z) = -\frac{\a-\b-|\a-\b|}{2\b},
\]
the atomic part of the measure appears only when $\a < \b$, has weight $1 - \frac{\a}{\b}$, and is concentrated at 0.\\
For the absolutely continuous part, we apply the Stieltjes inversion formula. Let $A + iB$ denote the argument of the square root in \eqref{CT1}, with $z = x + iy$. Then
\begin{align*}
A + iB &= x^4 + y^4 - 6 x^2 y^2 - 2 (\alpha+\beta)(x^2 - y^2) + (\alpha-\beta)^2\\
&+ 4 i (x^3 y - x y^3 - (\alpha+\beta) x y).
\end{align*}
Write $\sqrt{A+iB} = P + iQ$, where
\[
P := \frac{1}{\sqrt{2}} \sqrt{\sqrt{A^2 + B^2} + A}, \quad Q := \frac{\mathrm{sgn} B}{\sqrt{2}} \sqrt{\sqrt{A^2 + B^2} - A}.
\]
Then
\begin{align}\label{CT2}
\mathcal{G}_{\a,\b}(x+iy) = \frac{1}{2 \beta} \frac{(x^2 - y^2 - P - (\alpha-\beta)) + i(2xy - Q)}{x + iy}.
\end{align}
Now, $\di \widetilde{\nu}(x) = -\frac{1}{\pi} \lim_{y \to 0^+} \mathrm{Im} \, \mathcal{G}_{\a,\b}(x + iy)$. From \eqref{CT2}, the imaginary part is
\[
\frac{1}{2 \beta (x^2 + y^2)} \big[ (2xy - Q) x - (x^2 - y^2 - P - (\alpha-\beta)) y \big].
\]
Taking the limit $y \to 0^+$ reduces this to
\[
-\frac{1}{2 \beta x^2} Q x = -\frac{1}{2 \beta |x|} \sqrt{2 (\alpha+\beta) x^2 - (\alpha-\beta)^2 - x^4}.
\]
This quantity is nonzero only on the intervals $(-\gamma, -\eta)$ and $(\eta, \gamma)$, yielding the stated absolutely continuous part. One can check that
\[
\int_{-\infty}^\infty \di \widetilde{\nu}(x) =
\begin{cases}
\frac{\a}{\b} & \text{if $\a \le \b$,} \\
1 & \text{if $\a > \b$.}
\end{cases}
\]
\end{proof}
When $\a = \b$, $\nu$ reduces to the standard Wigner semicircle law with variance $\b$. Moreover, the vacuum law of the sum of $m$ position operators is obtained by dilating the measure $\nu$ by a factor of $m$.\\
In what follows, we shall denote by $\gs$ the $C^*$-algebra generated by position operators. As for the vacuum expectation, with an abuse of notation we denote by $\t_o$ the extension of the state from the corner algebras $Q\gs Q$ to the whole $\gs$.\\
The distributions of the position operators with respect to the tracial states $\t_o$ and $\t$ can be determined starting from the $\t_o$-moment description. We first note that Proposition \ref{sqfact} and an argument that overlaps with the one leading to \eqref{2ome01b} and \eqref{2ome02a} give, for each $n$ and $i_1,\ldots,i_n\in\bn$
\begin{align}\label{momo2}
\t_\Omega(s_{i_1}...s_{i_n})=
\begin{cases}
0 & \text{if $n$ is odd},\\
\displaystyle
\sum_{\{(l_h,r_h)\}\in NCPP(2m)}\prod_{h=1}^m \om_{2h-l_h}\delta_{i_{l_h},i_{r_h}}
& \text{if $n=2m$}\,,
\end{cases}
\end{align}
\begin{align}\label{momo3}
\t_o(s_{i_1}...s_{i_n})=
\begin{cases}
0 & \text{if $n$ is odd},\\
\displaystyle
\sum_{\{(l_h,r_h)\}\in NCPP(2m)}\prod_{h=1}^m \om_{2h-l_h+1}\delta_{i_{l_h},i_{r_h}}
& \text{if $n=2m$}\,.
\end{cases}
\end{align}
Thus, the $\t_o$-moments of the position operators coincide with the vacuum moments after exchanging $\a$ and $\b$. The same rule is then applied to the Cauchy transform, and consequently to the distribution $\nu_o$, which is again a free Poisson law.\\
For the trace $\t$ defined in \eqref{trac}, the Stieltjes inversion formula yields that, for each $j\in\bn$, the distribution of $s_j$ is $\displaystyle{\frac{1}{\a+\b}(\b\nu+\a\nu_o)}$, \textit{i.e.}
$$
\frac{|\a-\b|}{\a+ \b} \delta_0 +  \frac{1}{\pi (\a+\b) |x|} \sqrt{2(\alpha+\beta)x^2 - (\alpha-\beta)^2 - x^4}\, \chi_{(-\gamma, -\eta) \cup (\eta, \gamma)}(x) \di x,
$$
where $\g$ and $\eta$ are as in Proposition \ref{distri}.\\
We conclude the section by investigating free independence for the family of position operators with respect to the tracial states introduced above.\\
Recall that for a unital $C^*$-algebra $\gc$ and a state $\f$ on it, a family of $C^*$-subalgebras $(\gc_j)_{j\in J}$, where $J$ is an index set, endowed with the same unit of $\gc$, is free independent in the state $\f$ if for each $n\in\bn$ and each $j_1\neq j_2\neq \cdots \neq j_n$ in $J$, one has
$$
\f(a_1a_2\cdots a_n)=0\,, \quad a_i\in \gc_{j_i}
$$
when $\f(a_i)=0$ for any $i$. A family $(a_j)_{j\in J}$ in $\gc$ is free independent if the $C^*$-algebras generated by each $a_j$ and the unit yield a free independent family of subalgebras. In \cite[Section 3]{BLS}, a notion of $c$-free cumulant $R^{(c)}_k$, $k\in\bn$, for a pair of states on a $*$-algebra is given. They are defined recursively, and for the pair $(\t_o,\tau_\Omega)$ this yields
\begin{align}\label{wick0}
\begin{split}
\t_o(s_{i_1} \cdots s_{i_n})&=\sum_{k=0}^{n-1}
\;
\sum_{1 < l(1) < \cdots < l(k) \leq n}
R^{(c)}_{k+1}\big[s_{i_1}, s_{i_{l(1)}}, \ldots, s_{i_{l(k)}}\big]\\
\cdot& \t_\Om (s_{i_2} \cdots s_{i_{l(1)-1}})
\cdots
\t_\Om(s_{i_{l(k-1)+1}} \cdots s_{i_{l(k)-1}})
\;
\t_o(s_{i_{l(k)+1} } \cdots s_{i_{i_n}}).
\end{split}
\end{align}
Considering the equations \eqref{momo2} and \eqref{momo3} and the decomposition
of a pair partition according to its first block, the unique solution is  given by
$$
R_2^{(c)}[s_i,s_j]=\b \delta_{i,j}, \quad R_{n}=0 \textrm{ for } n\neq 2\,.
$$
This is a reminder of the freeness in case $\a=\b$. One can also do it for all
possible pairs in $\{\t,\t_o,\t_\Omega\}$.\\
In the general case of period-2 interacting Fock space, the position operators are not free independent in the vacuum state. This happens only if $\a=\b$, \textit{i.e.} in the (trivial deformation of the) full Fock case. Indeed, one has
$$
\t_\Om\big(s_1\big(s_2^2-\t_\Om(s_2^2)I\big)s_1\big)=\t_\Om\big(s_1\big(s_2^2-\a I\big)s_1\big)=\a(\b-\a)\,.
$$
By exchanging $\a$ and $\b$, one obtains the analogous statement for $\tau_0$ and again the same choice as above implies that, in general, there is no freeness in the state $\t$ defined as in \eqref{trac}, except the case $\a=\b$, where $\t_o$ collapses to $\t_\Om$. More in detail,
\begin{align*}
\t\big(s_1\big(s_2^2-\t(s_2^2)I\big)s_1\big)&= \t(s_1s_2^2s_1)- \t(s_1^2)^2
=\frac{\b^2\a+\a^2\b}{\a+b}-\bigg(\frac{2\a\b}{\a+\b}\bigg)^2\\
&=\a\b\bigg(1-\frac{2\a\b}{(\a+\b)^2}\bigg)=\frac{\a\b}{(\a+\b)^2}(\a-\b)^2\,.
\end{align*}
Finally, one can wonder if the family $(s_i^2)_i$ is free independent with respect to the vacuum state.  The following more general results answers the question in the affirmative sense. To this aim, denote by $\gd_j$ the $C^*$-closure of $\gd^0_j$. 
\begin{prop}\label{2perfree}
For a 1-mode type interacting Fock space $\cf(\ch,\l_n)$, the following hold:
\begin{itemize}
\item[(i)] The family $(s_j)_j$ is $\t_\Om$- free independent in the vacuum state if and only if then the weights $(\omega_n)_{n\geq 1}$ are constant.
\item[(ii)]  The family $(s_j^2)_j$ is $\t_\Om$-free if and only if the weights $(\om_n)_n$ have period 2. Moreover, if $(\om_{n})$ has period 2 , the family $(\gd_j)_j$ is $\t_\Om$-free independent.
\end{itemize}
\end{prop}
\begin{proof}
  For (i), we note that, if the sequence $(\omega_n)_{n\geq 1}$ is constant, then
  the family interacting Fock space reduces to the full Fock space based on the Hilbert space $\lambda_1\ch$. In this case the $*$-algebras generated by annihilators are $\t_\Om$-free.\\
Conversely, we first recall that if $x$ and $y$ are free in the vacuum, then
  $\tau_\Om(x^*y^*yx)=\tau_\Om(x^*x)\tau_\Om(y^*y)$, and hence $\|yx\Omega\|=\|x\Omega\|\|y\Omega\|$. Since for each $n$
$$
s_2^n\Om=e_2^{\otimes n}
+\sum_{k=0}^{n-1}c_{n,k}e_2^{\otimes k}\,, \quad c_{n,k}\in\br
$$
the vectors $\Om,s_2\Om,\ldots,s_2^n\Om$ and $\Om,e_2,\ldots,e_2^{\otimes n}$
span the same subspace. Therefore, there exists a unique monic polynomial
$p_n$ of degree $n$ such that $p_n(s_2)\Om=e_2^{\otimes n}$. We choose $x=s_1$ and $y=p_n(s_2)$ %where $P\in \mathbb R[X]$ so that
  %P(s_2)\Om=e_{2}^{\otimes n}$,
  $n\geq 1$. We get $\|e_{2}^{\otimes n}e_1\|=\|e_{2}^{\otimes n}\|\|e_1\|$, that is $\lambda_{n+1}=\lambda_n\lambda_1$. Hence, $\omega_n=\lambda_n/\lambda_{n-1}=\lambda_1$ for all $n\geq 1$.\\
Let us prove (ii). We start with the second part. Suppose that the weights are 2-periodic. For $\t_\Om$-freeness of $(\gd_j)_j$, it is sufficient to prove that for any $n\in\bn$, $i_1\neq i_2\neq \cdots\neq i_n\in\bn$, one has
\begin{equation}\label{freen}
\t_\Om\big((x_1-\t_\Om(x_1)I)(x_2-\t_\Om(x_2)I)\cdots (x_n-\t_\Om(x_n)I)\big)=0\,,
\end{equation}
where for each $j$, $x_j=(a_{i_j}^\dag)^{n_{+,j}}a_{i_j}^{n_{-,j}}\prod_{k=1}^{n_{0,j}}\om_{N+h_k}$ is a linear generator
of the unital $*$-algebra $\gd^0_{i_j}$ appearing in Proposition \ref{free1}. \\
We first suppose that there exists at least a word $x_k$ in the chain
above such that $n_{+,k}+n_{-,k}=0$, and take $m$ the largest index
for which this holds. Thus, $x_m$ has the form
$\displaystyle{\prod_{l=1}^r\om_{N+h_l}}$, and for each $p>m$,
$\t_\Om(x_p)=0$ as $n_{+,p}+n_{-,p}>0$. Hence,
$\big(x_{m+1}-\t_\Om(x_{m+1})I\big)\cdots\big(x_{n}-\t_\Om(x_{n})I\big)=x_{m+1}\cdots
x_n$, and denote
$$
\displaystyle{M_+(m):=\sum_{j=m+1}^n n_{+,j}}\,, \quad \displaystyle{M_-(m):=\sum_{j=m+1}^n n_{-,j}}\,.
$$
If $M_+(m)<M_-(m)$, then $x_{m+1}\cdots x_n\Om=0$, as there are more annihilators than creators. Therefore, \eqref{freen} follows. Since in general $M_+(m)-M_-(m)\in 2\bz$, the case $M_+(m)\geq M_-(m)$ implies $M_+(m)-M_-(m)=2p$, for $p\in\bn$. As a consequence, \eqref{cr2} gives
\begin{align*}
&(x_{m}-\t_\Om(x_{m})I\big)\big(x_{m+1}-\t_\Om(x_{m+1})I\big)\cdots\big(x_{n}-\t_\Om(x_{n})I\big)\Om\\
&= \bigg(\prod_{l=1}^r\om_{N+h_l}-\prod_{l=1}^r\om_{h_l}\bigg)x_{m+1}\cdots x_n\Om\\
&=x_{m+1}\cdots x_n\bigg(\prod_{l=1}^r\om_{N+h_l+2p}-\prod_{l=1}^r\om_{h_l}\bigg)\Om\,.
\end{align*}
Here,
$$
\displaystyle{\bigg(\prod_{l=1}^r\om_{N+h_l+2p}-\prod_{l=1}^r\om_{h_l}\bigg)\Om=\bigg(\prod_{l=1}^r\om_{h_l+2p}-\prod_{l=1}^r\om_{h_l}\bigg)\Om}=0\,,
$$
as $\om_{l+2p}=\om_l$ for each $l\in\bn^*$. Thus, also in this case, \eqref{freen} follows.\\
Suppose now that for each $k=1,\ldots, n$, one has $n_{+,k}+n_{-,k}>0$. This gives $x_k-\t_\Om(x_k)I=x_k$, as $x_k$ has mean zero. Therefore, the chain $(x_1-\t_\Om(x_1)I)(x_2-\t_\Om(x_2)I)\cdots (x_n-\t_\Om(x_n)I)$ reduces to $x_1\cdots x_n$. Again \eqref{cr2} allows to shift each $\prod_{k=1}^{n_0}\om_{N+h_k}$ to the right of the above chain. Thus, if $F:\bz\rightarrow \bc$, one simply denotes
$$
x_1\cdots x_n=(a^\dag_{i_1})^{n_{+,1}}a_{i_i}^{n_{-,1}}\cdots (a^\dag_{i_i})^{n_{+,n}}a_{i_n}^{n_{-,n}}F(N)\,,
$$
where $F(N)$ is achieved through functional calculus.
Since $F(N)\Om=F(0)\Om$, \eqref{freen} is proved if one gets
$$
\t_\Om((a^\dag_{i_1})^{n_{+,1}}a_{i_i}^{n_{-,1}}\cdots (a^\dag_{i_i})^{n_{+,n}}a_{i_n}^{n_{-,n}})=0\,.
$$
If $n_{-,n}>0$, the thesis is got to. Therefore, we reduce ourselves to showing that
\begin{equation}\label{freen2}
\t_\Om((a^\dag_{i_1})^{n_{+,1}}a_{i_i}^{n_{-,1}}\cdots (a^\dag_{i_i})^{n_{+,n}})=0\,,
\end{equation}
which trivially follows if $n_{-,k}=0$ for each $k\leq n-1$. In contrast, take $m$ the largest index for which $n_{-,m}>0$. As in this case also $n_{+,m+1}>0$, \eqref{cr} implies \eqref{freen2}.\\
Now assume that $(s_j^2)_j$ is $\tau_\Om$-free independent. First, iterating the arguments in the proof of Proposition \ref{moment1}, for each $\eps\in\{-1,1\}^{2n}_+$ one has
\begin{equation}\label{dep2}
\t_\Om\big(a_ia_ia^{\eps(1)}_j \cdots a^{\eps(2n)}_j a_i^\dag a_i^\dag\big)=\om_1\om_2\prod_{h=1}^n \om_{2h-l_h+2}\,.
\end{equation}
Freeness yields $\t_\Om(s_i^2s_j^{2}s_i^2)=\t_\Om(s_i^4)\t_\Om(s_j^{2})$, which immediately entails $\om_1(\om_2\om_3+\om_1^2)=\om_1(\om_1\om_2+\om_1^2)$, and therefore $\om_3=\om_1$. An inspection of $\t_\Om(s_i^2s_j^{4}s_i^2)=\t_\Om(s_i^4)\t_\Om(s_j^{4})$ gives
$$
\om_1\om_2(\om_3\om_4+\om_3^2)+\om_1^2(\om_1\om_2+\om_1^2)=(\om_1\om_2+\om_1^2)^2\,,
$$
which implies $\om_2=\om_4$, as $\om_1=\om_3$. Suppose now that, for each $2\leq j\leq n+1$, $\om_j=\om_1$ if $j$ is odd, and $\om_j=\om_{2}$ if $j$ is even. By \eqref{2ome01b} and \eqref{dep2} it turns out
\begin{align*}
\t_\Om\big(s_i^2s_j^{2n}s_i^2\big)&=\t_\Om\big(a_ia_is_j^{2n}a_i^\dag a_i^\dag\big)+\t_\Om\big(a_ia_i^\dag s_j^{2n}a_ia_i^\dag\big)\\
&=\om_1\om_2\sum_{\eps\in\{-1,1\}^{2n}_+}\prod_{h=1}^n \om_{2h-l_h+2}\,\, +\om_1^2\sum_{\eps\in\{-1,1\}^{2n}_+}\prod_{h=1}^n \om_{2h-l_h}
\end{align*}
and
$$
\t_\Om(s_i^4)\t_\Om(s_j^{2n})=(\om_1\om_2+\om_1^2)\sum_{\eps\in\{-1,1\}^{2n}_+}\prod_{h=1}^n \om_{2h-l_h}
$$
Since by freeness $\t_\Om(s_i^2s_j^{2n}s_i^2)=\t_\Om(s_i^4)\t_\Om(s_j^{2n})$,
$$
\sum_{\eps\in\{-1,1\}^{2n}_+}\prod_{h=1}^n \om_{2h-l_h+2}=\sum_{\eps\in\{-1,1\}^{2n}_+}\prod_{h=1}^n \om_{2h-l_h}
$$
As above, one reduces to the partition with the maximum depth. Indeed, all other non crossing pair partitions involve weights already determined by the induction hypothesis, so that the only contribution depending on $\om_{n+2}$ arises in this case. Here, $\om_3\cdots \om_{n+1}\om_{n+2}=\om_1\cdots \om_{n-1}\om_n$. If $n$ is even, then the induction assumption gives $\om_{n+1}=\om_1$, and therefore $\om_{n+2}=\om_2$. The contrary happens if $n$ is odd.
\end{proof}
As usual, we denote by $\mathbb F_\infty$ the free group with countable generators, and by $\l:\mathbb F_\infty\rightarrow \mathcal{B}(\ell^2(\mathbb F_\infty))$ its left regular representation. Recall that $L(\mathbb F_\infty)$ is the group von Neumann algebra on $\mathbb F_\infty$, namely $L(\mathbb F_\infty)=\{\l(g)\mid g\in \mathbb F_\infty\}''$.
\begin{cor}
The von Neumann algebra $\mathcal{S}$ contains a (non unital) copy of $L(\mathbb F_\infty)$. \end{cor}
\begin{proof}
By Proposition \ref{2perfree}, using that $P\Omega=P$ and $P$
commutes with $s_i^2$, the family $\{s_i^2P\mid i\in \mathbb N\}$ is
free in $\mathcal{S}_P$ (w.r.t. the trace $\t_\Om$), thus also free with respect to the trace
$\tau$. By Theorem \ref{MarPa}, for each $i\in\bn$, the vacuum distribution of
  $s_i^2P$ is diffuse if $\a>\b$, thus the von Neumann algebra $\{s_i^2P\mid i\in \mathbb N\}''$ is isomorphic to $L(\mathbb F_\infty)$ \cite{V} with unit $P$.\\
When $\b>\a$, one just needs to notice that the joint distribution of
  $\{s_i^2Q\}$ with respect to $\tau_o$ is the same as that of $\{s_i^2P\}$ with respect to $\tau_\Omega$ if one exchanges the role of $\a$ and $\b$. Thus, one can conclude as above.
\end{proof}
\begin{rem}
In general the von Neumann algebra $W^*(s_1)$ generated by $s_1$ is not a masa in $\mathcal S$. Indeed $q_i=1_{\{0\}}(s_i)$ is a projection with $\displaystyle{\tau(q_i)=\frac{|\a-\b|}{\a+\b}}$. If $2|\a-\beta|>\a+\b$, then $q=q_1\wedge q_2$ is a non zero
  projection as $\tau(1-q_1)+\tau(1-q_2)<1$. By Lemma \ref{kernel} the kernels of $s_1$ and $s_2$ are different subspaces, hence $q\neq q_1$. This means that $q$ is a nonzero proper subprojection of $q_1$, and therefore is not in $W^*(s_1)$. But clearly $q$ commutes with $s_1$, thus $W^*(s_1)$ is not maximal abelian.
\end{rem}
We end with some information about the vacuum law for sums and product of square of position operators.\\
Recall that the free cumulants of a free Poisson law with rate $\frac{\a}{\b}$ and size $\b$ are $r_m=\a\b^{m-1}$ for each $m\geq 1$ (see, \textit{e.g} . \cite[Lecture 11]{NS}). The additivity of the free cumulants given by freeness and Proposition \ref{MarPa} show that the vacuum law of $\displaystyle{\sum_{i=1}^n s_i^2}$ is again a free Poisson, with parameters $(\frac{n\a}{\b}, \b)$.\\
Finally, the $S$-transform \cite{NS} of each $s_i^2$ is $S(z)=\frac{1}{\a+\b z}$. Thus, the multiplicative property given by freeness implies that the vacuum law of $\displaystyle{\prod_{i=1}^n s_i^2}$ has the $S$-transform given by $\overline{S}(z)=\frac{1}{(\a+\b z)^n}$, \textit{i.e.} it is the Fuss-Catalan distribution of order $n$ when $\a=\b$ \cite{Ml}.

\section*{Acknowledgments}
\noindent The first-named author acknowledges the support of italian INDAM-GNAMPA

\end{document}